\documentclass[11pt]{amsart}
\usepackage{amssymb,latexsym,graphicx, amscd, enumitem}
\newtheorem{theorem}{Theorem}[section]
\newtheorem{prop}[theorem]{Proposition}
\newtheorem{lemma}[theorem]{Lemma}
\newtheorem{remark}[theorem]{Remark}
\newtheorem{question}[theorem]{Question}
\newtheorem{definition}[theorem]{Definition}
\newtheorem{cor}[theorem]{Corollary}
\newtheorem{example}[theorem]{Example}

\begin{document}

\title{The curve cone of almost complex $4$-manifolds}
\author{Weiyi Zhang}
\address{Mathematics Institute\\  University of Warwick\\ Coventry, CV4 7AL, England}
\email{weiyi.zhang@warwick.ac.uk}

 \begin{abstract}
 In this paper, we study the curve cone of an almost complex $4$-manifold which is tamed by a symplectic form. In particular, we prove the cone theorem as in Mori theory for all such manifolds using the Seiberg-Witten theory. For small rational surfaces and minimal ruled surfaces, we study the configuration of negative curves. We define abstract configuration of negative curves, which records the homological and intersection information of curves. Combinatorial blowdown is the main tool to study these configurations. As an application of our investigation of the curve cone, we prove the Nakai-Moishezon type duality for all almost K\"ahler structures on $\mathbb CP^2\#k\overline{\mathbb CP^2}$ with $k\le 9$ and minimal ruled surfaces with a negative curve. This is proved using a version of Gram-Schmidt orthogonalization process for the $J$-tamed symplectic inflation.
 \end{abstract}

  \maketitle

\section{Introduction}
The study of the curve cone is of much importance to the birational geometry of an algebraic variety. The cone theorem for smooth varieties proved in \cite{Mor82}, which describes the structure of the curve cone by extremal rays, was the first major step of Mori's program. It was later generalized to a larger class of varieties by Koll\'ar, Reid, Shokurov and others. The proof for a general variety relies on the bend-and-break technique, where a characteristic $0$ proof is still lacking. However, there is an elementary proof for algebraic surfaces, see {\it e.g.} \cite{Reidsur}. Early applications of the notion of the curve cone include Nakai-Moishezon's and Kleiman's ampleness criteria. 

 Recall that an almost complex structure $J$ is said to be {\it tamed} if there is a symplectic form $\omega$ such that the bilinear form $\omega(\cdot, J(\cdot))$ is positive definite. We could also similarly define the curve cone $A_J(M)$ for a tamed almost complex manifold $(M, J)$:
$$A_J(M)=\{\sum a_i[C_i]| a_i> 0\}$$where $C_i$ are irreducible $J$-holomorphic subvarieties on $M$. Here an irreducible $J$-holomorphic subvariety is the image of a $J$-holomorphic map $\phi: \Sigma\rightarrow M$ from a complex connected curve $\Sigma$, where $\phi$ is an embedding off a finite set. More generally, a $J$-holomorphic subvariety is a finite set of pairs $\{(C_i, m_i), 1\le i\le n\}$, where each $C_i$ is irreducible $J$-holomorphic subvariety and each $m_i$ is a non-negative integer. Later, we sometimes say $J$-holomorphic curves (or simply, curves) instead of $J$-holomorphic subvarieties.

We will focus on dimension $4$. Hence by taking Poincar\'e duality (we will identify the curve classes with their Poincar\'e dual cohomology classes by abusing the notation), we have $A_J(M)$ sitting as a cone in vector space $ H_J^+(M)\subset H^2(M; \mathbb R)$. Here $H_J^+(M)$ is called the $J$-invariant cohomology which is introduced in \cite{LZ, DLZ} along with the $J$-anti-invariant $H_J^-(M)$. The almost complex
structure $J$ acts on the bundle of real 2-forms $\Lambda^2$ as an
involution, by $\alpha(\cdot, \cdot) \rightarrow \alpha(J\cdot,
J\cdot)$. This involution induces the splitting into $J$-invariant, respectively,
$J$-anti-invariant 2-forms
$\Lambda^2=\Lambda_J^+\oplus \Lambda_J^-$. Then we define
\begin{equation*}
H_J^{\pm}(M)=\{ \mathfrak{a} \in H^2(M;\mathbb R) | \exists \;
\alpha\in  \Lambda_J^{\pm}, \, d\alpha=0 \mbox{ such that } [\alpha] =
\mathfrak{a} \}.
\end{equation*}

Usually, to prove a geometric result for a general (non-generic) tamed almost complex structure is far more delicate than the generic case, since we have to deal with very general subvarieties.  The techniques developed in \cite{LZ-generic, LZrc} enable us to have a fairly clear structural picture for subvarieties in a $J$-nef class. In this paper, we would develop more techniques to work with a general tamed almost complex structure. In particular, we find that the curve cone and sometimes curve configurations behave in many aspects similar to that of algebraic varieties. 

Throughout the paper, all the numerical calculations have the following geometric picture in mind: we think of an element in $A_J(M)$ (or $A^{\mathbb Q}_J(M)=A_J(M)\cap H^2(M, \mathbb Q)$) as the homology class of an $\mathbb R$ (or $\mathbb Q$) $J$-holomorphic subvarieties. Here an $\mathbb R$ (or $\mathbb Q$) subvariety allows $m_i$ in the definition of subvariety to be positive real numbers (or rational numbers). With this freedom in hand, we can choose some or all ``irreducible components" $C_i$ to be the extremal rays of $A_J(M)$. 
Here a subcone $N \subset A_J(M)$ is called {\it extremal} if $u,v \in A_J(M), u+v\in N$ imply that $u, v\in N$. A $1$-dimensional extremal subcone is called a {\it extremal ray}. Hence, the numerical information of the given class in $A_J(M)$ would give information for these extremal rays, and vice versa. The geometry is  thus hidden in the numerical information. A typical example is Corollary \ref{-1} which guarantees the existence of $-1$ rational curves by the numerical property of a class known to be in $A_J(M)$.

From another viewpoint, considering $\mathbb Q$-subvarieties of an integral class $e\in A_J(M)$ is the same as considering subvarieties in all the integral classes of the ray $\mathbb R^+ \cdot e\subset A_J(M)$. The complexity arises, thus one has to study all integral classes in the ray rather than just a single class, because the subvarieties in a class and its multiples do not change linearly, see {\it e.g.} Examples \ref{-K8blowup}, \ref{ruledex}.

For this sake, we will first understand the extremal rays. Our first result, the cone theorem for tamed almost complex $4$-manifolds, is a general structural result on the extremal rays of the ``negative" part (this is not to be confused with the notion of ``negative curves" used later, which means curves with negative self-intersection). There are several different proofs of the cone theorem in algebraic geometry (at least for algebraic surfaces). Mori's original proof uses bend-and-break and Kleiman's ampleness criterion. The argument  in \cite{Reidsur} is bend-and-break free, but it proves rationality theorem first. Our proof proceeds in a totally different logical order. We do not need all these technical results before the cone theorem.  In our paper, the rationality theorem, Proposition \ref{ratthm}, is proved as a corollary of the cone theorem. Moreover, we are only able to prove Nakai-Moishezon and Kleiman type results for certain rational surfaces.

 However, the soul of the algebraic geometric proofs and our proof is the same: vanishing $\Rightarrow$ non-vanishing $\Rightarrow$ cone theorem. 
We start by proving certain Seiberg-Witten invariants vanish.  Using the  Seiberg-Witten wall crossing formula \cite{LLwall},  we have non-vanishing results. Then Taubes' SW=Gr \cite{T, T3} would imply a class does not span an extremal ray if it could be written as the sum of two classes with non-trivial Seiberg-Witten invariants.  This would eventually guarantee the extremal rays of the ``negative" part are generated by rational curves. 
The precise statement of the cone theorem is the following, which is a combination of Theorem \ref{conethm} and Proposition \ref{exray}. We use $\mathbb R^+$ to denote the interval $[0, \infty)$.
\begin{theorem}\label{coneintro}
Let $(M, J)$ be a tamed almost complex $4$-manifold. Then $$\overline{A}_J(M)=\overline{A}_J^{K_J\ge 0}(M)+\sum \mathbb R^+[L_i]$$ where $L_i\subset M$ are countably many smooth irreducible rational curves such that $-3\le K_J\cdot [L_i]<0$ which span the extremal rays $\mathbb R^+[L_i]$ of $\overline{A}_J(M)$.

Moreover, for any $J$-almost K\"ahler symplectic form $\omega$ and any given $\epsilon>0$, there are only finitely many extremal rays with $(K_J+\epsilon[\omega])\cdot [L_i]\le 0$.

In addition, an irreducible curve $C$ is an extremal rational curve if and only if 

\begin{enumerate}
\item $C$ is a $-1$ rational curve;
\item $M$ is a minimal ruled surface or $\mathbb CP^2\# \overline{\mathbb CP^2}$, and $C$ is a fiber;
\item $M=\mathbb CP^2$ and $C$ is a projective line.
\end{enumerate}
\end{theorem}

In fact, we prove a slightly stronger version of cone theorem:
$$\overline{A}^{\mathbb Q}_J(M)=\overline{A}_J^{K_J\ge 0}(M)\cap H^2(M, \mathbb Q)+\sum \mathbb Q^+[L_i],$$where $\overline{A}^{\mathbb Q}_J(M)=\overline{A}_J(M)\cap H^2(M, \mathbb Q)$. 

Here $K_J$ is the canonical class of $(M, J)$ and $A_J^{K_J\ge 0}(M)=\{C\in A_J(M)| K_J\cdot C\ge 0\}$ is the ``positive" part of the curve cone. In general, this part is not generated by countably many extremal rays since it may have round boundary. This phenomenon happens in particular when we do not have sufficient curves as for a generic almost complex structure on manifolds with $b^+>1$, or even in the rational surfaces $\mathbb CP^2\#k\overline{ \mathbb CP^2}$ with $k> 9$.

On the other hand, there are indeed many cases whose curve cones are polytopes. The most well known examples are rational surfaces $\mathbb CP^2\#k\overline{ \mathbb CP^2}$ when $k\le 9$ and $S^2\times S^2$. The cases of $S^2$ bundles over $S^2$ are treated in \cite{LZ-generic} for example. 

We give a careful analysis of the curve cone for all possible tamed almost complex structures on $\mathbb CP^2\#2\overline{ \mathbb CP^2}$ in Section $3$. In particular, we show that there are at least two embedded $-1$ rational curves for any tamed almost complex structures (Theorem \ref{cp2+2}). This should be compared to other results on the existence of embedded $-1$ rational curves. As a corollary of Theorem \ref{coneintro}, we give a proof of the fact that any tamed $J$ on a non-minimal symplectic $4$-manifold which is not diffeomorphic to a $S^2$ bundle contains at least one smooth $-1$ rational curve. In fact, it appears as a component of a $\mathbb Q$-variety representing an exceptional sphere class (Corollary \ref{-1}).  On the other hand, based on an argument communicated to the author by Dusa McDuff, we show that if $M$ is not diffeomorphic to one point blow up of $S^2$ bundles, then there exists a tamed $J$ such that all $-1$ rational curves are disjoint and any possible numbers of such $-1$ rational curves could be realized (Theorem \ref{disjoint-1}).

Let $h_J^+(M)=\dim H_J^+(M)$. Then by the light cone lemma, when $h_J^+(M)=b^-(M)+1$ and $b^-(M)>1$, the boundary hyperplanes of the dual of the curve cone are determined by $J$-holomorphic curves with negative self-intersection, {\it i.e.} negative curves (see Lemma \ref{negbdy}). The equality $h_J^+(M)=b^-(M)+1$ holds for any complex structure $J$ on manifolds with $b^+(M)=1$. Hence to determine various cones, {\it e.g.} the curve cone and the almost K\"ahler cone, for $\mathbb CP^2\#k\overline{ \mathbb CP^2}$ when $k\le 9$, we are reduced to determine all possible negative curves. The classification is done in Section $4$. Especially, all negative curves (resp. non-positive curves) on $\mathbb CP^2\#k\overline{ \mathbb CP^2}$ with $k\le 9$ (resp. $k<9$) are rational curves. A complete list of the classes of all such rational curves are obtained in Propositions \ref{a<0}, \ref{a>0}, \ref{a>0'}. Moreover, we have the following more precise statement about negative curve configurations.

\begin{theorem}\label{main1}
For rational $4$-manifolds $\mathbb CP^2\# k\overline{\mathbb CP^2}$ with $k<6$ or $S^2\times S^2$, the set of all the possible configurations of negative self-intersection curves for tamed almost complex structures are the same as the set for complex structures.
\end{theorem}

Namely, given  a tamed almost complex structures, we could find a complex structure (on the same manifold) such that their configurations of negative curves are the same. Here by a configuration of negative curves, we mean the set of homology classes and information on the multiplicities of each intersection point. See section 4 for a precise definition.

This is also true for minimal ruled surfaces. However, as observed in \cite{CP}, this is not true for a non-minimal ruled surface, {\it e.g.} $(T^2\times S^2)\# \overline{\mathbb CP^2}$.  A generic almost complex structure will have only two negative curves $E$ and $F-E$ where $F$ is class of $S^2$ and $E$ is the class of exceptional curve, while there is no complex structure on it having exactly those two negative curves. The paper \cite{CP} also gives such an example which is a minimal surface of general type. Currently, no such simply connected examples are known.  On $\mathbb CP^2\# k\overline{\mathbb CP^2}$ with $k\ge 9$, it is related to the Nagata conjecture.

\begin{question}(Nagata)
For every $k> 9$, is it true that for any complex structure on $\mathbb CP^2\# k\overline{\mathbb CP^2}$, we have
$$d> \frac{\sum_{q=1}^km_q}{\sqrt{k}},$$
for every irreducible curve $C$ with  $[C]=dH-\sum_{q=1}^km_qE_q$?
\end{question}

This is equivalent to say that $H-\frac{1}{\sqrt{k}}\sum E_i$ is on the closure of the K\"ahler cone by Nakai-Moishezon criterion. It is easy to see that a generic tamed almost  complex structure satisfies the inequality. 

 The above discussion gives us a clearer picture on the polytopic boundary of the curve cone $A_J(M)$. We can apply it to understand the Nakai-Moishezon or the Kleiman type duality between the curve cone and the almost K\"ahler cone $\mathcal K_J^c=\{[\omega]\in H^2(M; \mathbb R)| \omega \mbox{ is compatible with }J\}$. When $b^+(M)=1$, it is shown in \cite{LZ} that $\mathcal K_J^c$ is equal to the tame cone $\mathcal K_J^t=\{[\omega]\in H^2(M; \mathbb R)| \omega \mbox{  tames }J\}$. 

Let us introduce a couple more cones. First is the positive cone $\mathcal P=\{e\in H^2(M;\mathbb R)|e\cdot e >0\}$. The second cone $A_J^{\vee, >0}(M)$ (resp. $\overline{A}_J^{\vee, >0}(M)$) is the positive dual of $A_J(M)$ (resp. $\overline{A}_J(M)$) where the duality is taken within $H_J^+(M)$. Let $\mathcal P_J=A_J^{\vee, >0}(M) \cap \mathcal P$. Clearly,  $\mathcal K_J^c\subset \mathcal P_J$. Then we can ask the following

\begin{question} \label{dualC}
For an almost K\"ahler structure $J$ on  a closed, oriented $4$-manifold $M$ with $b^+(M)=1$,  is
\begin{equation*}\label{dualak}
\mathcal K_J^c=\mathcal P_J=\overline{A}_J^{\vee, >0}(M)  ?
\end{equation*}

When $b^+(M)>1$, is $\mathcal K_J^c$ a connected component of $\mathcal P_J$?
\end{question}

When $J$ is an integrable complex structure, $\mathcal K_J^c$ is the K\"ahler cone. A K\"ahlerian Nakai-Moishezon theorem is given by Buchdahl and Lamari in
dimension $4$ \cite{Bu, Lamari}, and Demailly-Paun \cite{DP}  in arbitrary dimension to determine the K\"ahler cone completely.

It is worth noting that in algebraic geometry, all the proofs of the cone theorem relies on the Nakai-Moishezon or Kleiman ampleness criterion. However, our almost complex cone theorem (Theorem \ref{coneintro}) does not need the corresponding version, {\it i.e.} Question \ref{dualC}.

The key of Question \ref{dualC} is to construct almost K\"ahler forms. There are currently two methods of construction. The first is using Taubes' subvarieties-current-form technique \cite{T2009}. However, to argue it for an arbitrary $J$, we have to use spherical subvarieties as in \cite{LZ-generic}. It was successfully used to affirm Question \ref{dualC} for $S^2$ bundles over $S^2$. The method is applied to prove it for $\mathbb CP^2\#2\overline{\mathbb CP^2}$ in this paper. However, it cannot go further along this line. The limit of this method is discussed in Section \ref{cp2+3}.

The second method is reducing the construction of the almost K\"ahler cone to determine the more flexible tame cone $\mathcal K_J^t$ by virtue of the identity $\mathcal K_J^c=\mathcal K_J^t\cap H_J^+(M)$ established in \cite{LZ}. In this situation, we could apply the $J$-tamed symplectic inflation developed by McDuff \cite{Mc2} and Buse \cite{Bus}. More precisely, we apply Buse's inflation for negative $J$-holomorphic curves to our extremal ray of curve cone. The combinatorial version of it, the formal $J$-inflation, is introduced. A version of Gram-Schmidt orthogonalization process is applied in our situation to significantly simplify the otherwise too complicated calculation. We are able to prove the following general result when the curve cone has no round boundary.

\begin{theorem}\label{noround}
On a $4$-dimensional almost complex manifold $(M,J)$ with $\mathcal K_J^c\ne \emptyset$, and $h_J^+=b^-+1$, if $\mathcal P_J$ has no round boundary and each of the boundary hyperplane is determined by a smooth negative curve, then $\mathcal K_J^c$ is a connected component of $\mathcal P_J$.
\end{theorem}

Applying the discussion of the curve cone for rational and ruled surfaces in Sections 2-4, we have the both Nakai-Moishezon and Kleiman dualities.

\begin{theorem}\label{rr}
If $J$ is almost K\"ahler on $M=\mathbb CP^2\# k\overline{\mathbb CP^2}$ with $k\le 9$ or a minimal irrational ruled surface with a negative curve, then $$\mathcal K_J^c=\mathcal K_J^t= \mathcal P_J=\overline{A}_J^{\vee, >0}(M).$$
\end{theorem}

We are also able to show that Question \ref{dualC} is true for a generic tamed $J$ on a manifold with $b^+(M)=1$.
Here a generic tamed almost complex structure means that it is chosen from a residual subset of all tamed almost complex structures. 
\begin{theorem}\label{generic}
On a symplectic $4$-manifold with $b^+=1$, $\mathcal K_J^t=\mathcal K_J^c= \mathcal P_J$ for a generic tamed $J$.
\end{theorem}

 The author is grateful to Tian-Jun Li and Clifford Taubes for very helpful suggestions on Theorem \ref{coneintro}. He is also in debt to Tian-Jun Li for pointing out reference \cite{Bus} and Dusa McDuff for many useful comments to greatly improve the presentation of the paper. Finally, we would also like to express deep gratitude and respect to the referees for meticulous reading and very helpful suggestions which improve the paper tremendously. 

The work is partially supported by AMS-Simons travel grant and EPSRC grant EP/N002601/1.

\section{The Cone Theorem}
In this section, we will give a proof of the cone theorem for tamed almost complex $4$-manifolds. Recall an almost complex structure $J$ is said to be tamed if there is a symplectic form $\omega$ such that the bilinear form $\omega(\cdot, J(\cdot))$ is positive definite. A taming form of $J$ is said to be compatible with $J$ if the bilinear form $\omega(\cdot, J(\cdot))$ is symmetric. If there is a symplectic form compatible with the almost complex structure $J$, then $J$ is called almost K\"ahler. Hence, a $J$-compatible symplectic form is also called a $J$-almost K\"ahler form.

Given a class $e\in H_2(M; \mathbb Z)$, introduce the $J$-genus of $e$, $$g_J(e)=\frac{1}{2}(e\cdot e+K\cdot e)+1,$$where $K=K_J$ is the canonical class of $J$. 
Moreover, when $C$ is an irreducible subvariety,  $g_J([C])$ is non-negative.  In fact, if $\Sigma$
is the model curve of $C$, by the adjunction inequality,  
\begin{equation}\label{adj inequality}
g_J(e_C)\geq g(\Sigma),
\end{equation}
with equality if and only if $C$ is smooth. 
In particular, if $C$ is the class of an irreducible curve having $C^2<0$ and $K\cdot C<0$ then the curve must be a $-1$ rational curve.

 In the following, a rational curve means an irreducible $J$-holomorphic subvariety of $J$-genus $0$. By \eqref{adj inequality}, such a curve has to be smooth.
\subsection{The Seiberg-Witten invariant}
In this subsection, we will give a very brief introduction to the Seiberg-Witten invariant, which will be the main tool to establish Theorem \ref{coneintro}. For a detailed introduction, see for example \cite{LLwall, LLb+=1} and references therein.

Let $M$ be an oriented $4$-manifold with a given Riemannian metric $g$ and a spin$^{c}$ structure $\mathcal L$ on $M$.  Hence there are a pair of rank $2$ complex vector bundles $S^{\pm}$ with isomorphisms $\det(S^+)=\det(S^-)=\mathcal L$. The Seiberg-Witten equations are for a pair $(A, \phi)$ where $A$ is a connection of $\mathcal L$ and $\phi$ is a section of $S^+$. These equations are 
$$D_A\phi=0$$
$$F_A^+=iq(\phi)+i\eta$$ where $q$ is a canonical map $q: \Gamma(S^+)\rightarrow \Omega^2_+(M)$ and $\eta$ is a self-dual $2$-form on $M$. 

The group $C^{\infty}(M; S^1)$ naturally acts on the space of solutions and acts freely at irreducible solutions. Recall a reducible solution has $\phi=0$, and hence $F_A^+=i\eta$. The quotient is the moduli space and is denoted by $\mathcal M_M(\mathcal L, g, \eta)$. For generic pairs $(g, \eta)$, the Seiberg-Witten moduli space $\mathcal M_M(\mathcal L, g, \eta)$ is a compact manifold of dimension $$2d(\mathcal L)=\frac{1}{4}(c_1(\mathcal L)^2-(3\sigma(M)+2\chi(M)))$$ where $\sigma(M)$ is the signature and $\chi(M)$ is the Euler number. Furthermore, an orientation is given to $\mathcal M_M(\mathcal L, g, \eta)$ by fixing a homology orientation for $M$, {\it i.e.} an orientation of $H^1(M)\oplus H^2_+(M)$. When $b^+(M)=1$, the space of $g$-self-dual forms $\mathcal H^+_g(M)$ is spanned by a single harmonic $2$-form $\omega_g$ of norm $1$ agreeing with the homology orientation.

Now, for the convenience of notation, we fix a (homotopy class of) almost complex structures. In particular, it determines a canonical class $K$ which is the first Chern class of the cotangent bundle. We denote $e:=\frac{c_1(\mathcal L)+K}{2}\in H^2(M; \mathbb Z)/(2\hbox{-torsion})$. 
In the following, notice the first Chern class $c_1(\mathcal L)$ determines $\mathcal L$ and {\it vice versa}.
For a generic choice of $(g, \eta)$, the Seiberg-Witten invariant $SW_{M, g, \eta}(e)$ is defined as follows. If $d(\mathcal L)<0$, then the SW invariant is zero. If $d(\mathcal L)=0$, then the moduli space is a finite union of signed points and the SW invariant is the sum of the corresponding signs. If $d(\mathcal L)>0$, then the SW invariant is obtained by pairing the fundamental class of $\mathcal M_M(\mathcal L, g, \eta)$ with the maximal cup product of the Euler class of the $S^1$-bundle $\mathcal M^0_M(\mathcal L, g, \eta)$ over $\mathcal M_M(\mathcal L, g, \eta)$. Here $\mathcal M^0_M(\mathcal L, g, \eta)$ is called the based moduli space which is the quotient of the space of solution by $C_0^{\infty}(M, S^1)$ (the elements in $C^{\infty}(M, S^1)$ which map a base point in $M$ to $1$ in $S^1$).

If $b^+>1$, a generic path of $(g, \eta)$ contains no reducible solutions. Hence, the Seiberg-Witten invariant is an oriented diffeomorphism invariant in this case. Hence we can use the notation $SW(e)$ for the Seiberg-Witten invariant. We will also write $$\dim_{SW}(e)=2d(\mathcal L)=\frac{1}{4}((2e-K)^2-K^2)=e^2-K\cdot e$$ for the Seiberg-Witten dimension. In the case $b^+=1$, there might be reducible solutions on a $1$-dimensional family.  Recall that the curvature $F_A$ represents the cohomology class $-2\pi ic_1(\mathcal L)$. Hence $F_A^+=i\eta$ holds only if $-2\pi c_1(\mathcal L)^+=\eta$. This happens if and only if the discriminant $\Delta_{\mathcal L}(g, \eta):=\int (2\pi c_1(\mathcal L)+\eta)\omega_g=0$. With this in mind, the set of pairs $(g, \eta)$ with positive (resp. negative) discriminant is called the positive (resp. negative) $\mathcal L$ chamber. We use the notation $SW^{\pm}(e)$ for the Seiberg-Witten invariants in these two chambers. Moreover, in the this paper, we will use $SW(e)$ instead of $SW^-(e)$ when $b^+=1$. 

As one can easily see from the discussion above, $SW$ (or $SW^{\pm}$) could be defined more generally as map from $H^2(M, \mathbb Z)$ to $\Lambda^*H^1(M, \mathbb Z)$, although we do not need this generality for this paper. 

We now assume $(M, \omega)$ is a symplectic $4$-manifold, and $J$ is a $\omega$-tamed almost complex structure. Then the theorems in \cite{T3} and \cite{LLb+=1} equate Seiberg-Witten invariants with Gromov-Taubes invariants that are defined by making a suitably counting of $J$-holomorphic subvarieties (In fact, the $SW(e)$ we defined here is essentially the Gromov-Taubes invariant in the literature). Especially, when $SW(e)\ne 0$, there is a $J$-holomorphic subvariety in class $e$ passing through $\dim_{SW}(e)$ given points. This is the key result from the Seiberg-Witten theory we will use in this section.
\subsection{The cone theorem}
We first have the following

\begin{lemma}\label{Kod>0}
Let $J$ be a tamed almost complex structure on a symplectic $4$-manifold which is not rational or ruled. Let $C$ be an irreducible $J$-holomorphic curve with $K_J\cdot [C]<0$. Then $C$ has to be a $-1$ rational curve.
\end{lemma}
\begin{proof}
If $M$ is a symplectic $4$-manifold which is not rational or ruled, then by \cite{T, LL} we always have a $J$-holomorphic subvariety in class $2K_J$. If $2K_J=\sum a_i[C_i]$, where $a_i\ge 0$ and $C_i$ are irreducible curves, then $K_J\cdot C<0$ would imply $C$ is one of $C_i$. Let it be $C_1$. Then $[C]\cdot (2K_J-a_1[C])\ge 0$ implies $C^2<0$. Altogether, we have $C$ is a $-1$ rational curve. 
%
%
\end{proof} 

Next we prove a very useful lemma. In many applications later, our $C$ also has $SW([C])\ne 0$.

\begin{lemma}\label{int>0}
If $C$ is an irreducible $J$-holomorphic curve with $C^2\ge 0$ and $SW(e)\ne 0$, then $e \cdot [C]\ge 0$.
\end{lemma}
\begin{proof}

Since $SW(e)\ne 0$, we can represent $e$ by a possibly reducible $J$-holomorphic subvariety. Since each irreducible curve $C'$ has $[C']\cdot [C]\ge 0$, we have $e\cdot [C]\ge 0$.
\end{proof}

\begin{prop}\label{SW}
Let $(M, J)$ be a tamed almost complex $4$-manifold. Let $C$ be an irreducible curve such that $K_J\cdot [C]<0$. Then $SW([C])\ne 0$. Moreover, there is a curve in class $[C]$ passing through any given point when $C$ is not a $-1$ rational curve.
\end{prop}
\begin{proof}
First we could assume $M$ is rational or ruled. Otherwise, it is proved in Lemma \ref{Kod>0}.

Notice $C^2+K_J\cdot [C]=2g_J([C])-2$, $$[C]\cdot (K_J-[C])=-\dim_{SW} [C]=-\dim_{SW}(K_J-[C])=2K\cdot [C]+2-2g_J([C])\le 0.$$ The equality holds if and only if $g_J([C])=0, K_J\cdot [C]=-1$, and $C^2=-1$, {\it i.e.} $C$ is a $-1$ rational curve. Otherwise $C^2\ge 2g_J([C])-2+2\ge 0$.

Hence we assume $[C]\cdot (K_J-[C])<0$. If $SW(K_J-[C])\neq 0$, by SW=Gr we have a (possibly reducible) $J$-holomorphic curve in class $K_J-[C]$. Hence the irreducible curve $C$ must be a component of this curve and $C^2<0$. This contradicts the adjunction and dimension formula. Thus $SW(K_J-[C])=0$.

Since $\dim_{SW} [C]\ge 0$, we have wall-crossing formula
$$|SW(K-[C])-SW([C])|=\begin{cases}   1 & \hbox{if $(M, \omega)$ rational,}\cr
  |1+[C]\cdot T|^h &\hbox{if
  $(M, \omega)$ irrationally ruled,}\cr\end{cases}$$
  where $T$ is the unique positive fiber class and $h$ is the genus of base surface of irrationally ruled manifolds (see \cite{LLwall}).
Since $[C]\cdot T\ge 0$ by Lemma \ref{int>0}, we have $SW([C])\neq 0$ by the wall-crossing and $SW(K_J-[C])=0$. Hence there is a curve in class $[C]$ passing through any given point since $\dim_{SW} [C]>0$. 
\end{proof}

Next we will prove the first two statements of Theorem \ref{coneintro}, the cone theorem. 

\begin{theorem}\label{conethm}
Let $(M, J)$ be a tamed almost complex $4$-manifold. Then $$\overline{A}_J(M)=\overline{A}_J^{K_J\ge 0}(M)+\sum \mathbb R^+[L_i]$$ where $L_i\subset M$ are countably many smooth irreducible rational curves such that $-3\le K_J \cdot [L_i]<0$ which span the extremal rays $\mathbb R^+[L_i]$ of $\overline{A}_J(M)$.

Moreover, for any $J$-almost K\"ahler symplectic form $\omega$ and any given $\epsilon>0$, there are only finitely many extremal rays with $(K_J+\epsilon[\omega])\cdot [L_i]\le 0$.
\end{theorem}
\begin{proof}
We first prove our statement about the extremal rays $\mathbb R^+[L_i]$. 
When $M$ is not rational or ruled, then Lemma \ref{Kod>0} verifies our claim. Especially, $L_i$ are finitely many rational curves with $K_J\cdot [L_i]=-1$. 

In general, let $C$ be an irreducible curve with $K_J\cdot [C]<0$. By Proposition \ref{SW}, we have $SW([C])\neq 0$. Hence for any tamed almost complex structure $J$, there is a (possibly reducible) subvariety in class $[C]$ by Taubes' SW=Gr and Gromov compactness. Especially, it is true for a projective variety. 

Now, assume $C$ is an irreducible curve with $g_J([C])>0$ and $K\cdot [C]<0$, or  $g_J([C])=0$ and $K_J\cdot [C]< -3$ on a rational or ruled surface. We want to show that $[C]$ cannot span an extremal ray of the curve cone (we will say $[C]$ is not extremal for simplicity). We divide our discussion into the following cases.
\bigskip

\noindent {\bf Case 1: Irrational ruled surfaces} \\

$\bullet$ $M=\Sigma_h\times S^2$, $h> 1$, and its blowups \\

In this case, let $U$ be the class of the base $\Sigma_h$ and $T$ be the class of the fiber $S^2$. Then the canonical class $K_J=-2U+(2h-2)T+\sum_i E_i$. Let $[C]=aU+bT-\sum_i c_iE_i$.

Since both $[C]$ and $T$ pair negatively with $K_J$, by Lemma \ref{SW}, both classes have non-trivial Seiberg-Witten invariant. Applying Lemma \ref{int>0} to the pair $[C]$ and $T$, we have $a\ge 0$. We could assume $[C]$ is not one of the classes of exceptional curves $E_i$. Then applying Lemma \ref{int>0} to $[C]$ and $E_i$ gives $c_i\ge 0$.

The assumption $K_J\cdot [C]<0$ reads as $$a(2h-2)-2b+\sum c_i<0.$$Especially, we have $b>0$.

When $a=0$, we have $$-\sum c_i^2=C^2=2g_J([C])-2-K_J\cdot [C]>2g_J([C])-2\ge -2.$$ It works only when $g_J([C])=0$ and there is at most one nonzero $c_i$, say $c_1$, which equals to $1$. Hence, $K_J\cdot [C]=-1$ (resp. $K_J\cdot [C]=-2$), which would imply $[C]=T-E_1$ (resp. $[C]=T$). This finishes the case of $a=0$ by showing that all irreducible curves, in particular the curves span the extremal rays, are rational curves with $-2\le K_J\cdot [C]\le -1$.

When $a>0$, we take the projection $f: C\rightarrow \Sigma_h$ to the base. It has degree $a=[C]\cdot T$.
Since $\Sigma_h$ has genus greater than one, and by Kneser's theorem, we have 
$$2g_J([C])-2\ge a(2h-2)\ge 2a.$$

Now, we are planning to show that the class $[C]-T$ has non trivial Seiberg-Witten invariant. First we show $SW(K_J-([C]-T))=0$. If not, notice $SW(T)\ne 0$, we should have $(K_J-([C]-T))\cdot T\ge 0$  by Lemma \ref{int>0}. However, it contradicts to the calculation $(K_J-([C]-T))\cdot T=-2-[C]\cdot T<0$. Next we show the Seiberg-Witten dimension of it is nonnegative.

\begin{equation*} \label{SWdim}
\begin{array}{lll}
 \dim_{SW}([C]-T)&=&([C]-T)^2-K_J\cdot ([C]-T)\\
 &&\\
 &=&C^2-K_J\cdot [C] -2-2[C]\cdot T\\
 &&\\
 &\ge&2g_J([C])-2-2a\\
 &&\\
 &\ge&0.\\
 \end{array}
 \end{equation*}

Finally the wall crossing formula implies $$|SW([C]-T)|=|SW(K_J-([C]-T))-SW([C]-T)|=|1+a|^h\ne 0.$$

Hence, we complete our argument that $[C]=([C]-T)+T$ is not extremal in this case.
\\

$\bullet$ $M=T^2\times S^2$ and its blowups \\

We use the same setting as the above case. That is, we assume $[C]=aU+bT-\sum_i c_iE_i$ and the canonical class $K_J=-2U+\sum_i E_i$. 
As showed in the above case, we only need to show that when $a>0$, $[C]$ is not extremal. Without loss, we assume $c_1\ge c_2 \ge \cdots$.

Notice $$C^2=2ab-\sum c_i^2, \quad -K_J\cdot [C]=2b-\sum c_i>0.$$ 
We will first show that $[C]$ must not be extremal in the case that $c_1\le a$ and make a corresponding change in the $c_1>a$ case later.
$$C^2=2ab-\sum c_i^2\ge 2ab-a\sum c_i=a(2b-\sum c_i)=a(-K_J\cdot [C]).$$ 

Hence look at the classes $l[C]-T$, we have 

\begin{equation*} \label{SWdim4}
\begin{array}{lll}
 \dim_{SW}(l[C]-T)&=&(l[C]-T)^2-K_J\cdot (l[C]-T)\\
 &&\\
 &=&(l^2C^2-lK_J\cdot [C]) -2-2l [C]\cdot T\\
 &&\\
 &\ge&(l^2a+l)(-K_J\cdot [C])-2-2la\\
 &&\\
 &\ge& l^2a+l-2la-2\\
 \end{array}
 \end{equation*}
 It is greater than $0$ if $l$ is large enough ({\it e.g.} $l>2a$). 

Since $(K_J-(l[C]-T))\cdot T=-2-la<0$, Lemma \ref{int>0} implies $SW(K_J-(l[C]-T))= 0$. Apply the wall crossing formula $$|SW(l[C]-T)|=|SW(K_J-(l[C]-T))-SW(l[C]-T)|=|1+la|\ne 0.$$

Hence $[C]=\frac{1}{l}((l[C]-T)+T)$ where $l[C]-T$ is not proportional to $T$ since $a>0$. It is a decomposition of two non-proportional classes with non-trivial Seiberg-Witten invariant which is not extremal.

If $c_1>a$, 
 similar to the above, we first calculate the Seiberg-Witten dimension of the class $[C]-T+E_1$.

\begin{equation*} \label{SWdim2}
\begin{array}{lll}
 \dim_{SW}([C]-T+E_1)&=&([C]-T+E_1)^2-K_J\cdot ([C]-T+E_1)\\
 &&\\
 &=&(C^2-K_J\cdot [C]) -2-2[C]\cdot T+2c_1\\
 &&\\
 &\ge&C^2-K_J\cdot [C]\\
 &&\\
 &>&0.\\
 \end{array}
 \end{equation*}

Again $(K_J-([C]-T+E_1))\cdot T=-2-a<0$ implies $SW(K_J-([C]-T+E_1))=0$. Then  $$|SW([C]-T+E_1)|=|SW(K_J-([C]-T+E_1))-SW([C]-T+E_1)|=|1+a|\ne 0.$$ And $[C]=([C]-T+E_1)+(T-E_1)$ is not extremal.
\bigskip

$\bullet$ $M$ is a non-trivial  $S^2$ bundle over $\Sigma_h$, $h\ge 1$ \\

Since the blow-ups of it are diffeomorphic to those of trivial bundle, we are only left with the non-trivial bundles.

Let $U$ be the class of a section with $U^2=1$ and $T$ be the class of the fiber. Then $K=-2U+(2h-1)T$. Let $[C]=aU+bT$. Again, we have $a\ge 0$. The condition $K\cdot [C]<0$ reads as $$2b>a(2h-3).$$

If $a=0$, $$0=C^2=2g_J([C])-2-K_J\cdot [C]>2g_J([C])-2.$$Thus $g_J([C])=0$ and $K_J\cdot [C]=-2$. This completes the case of $a=0$ by showing that all irreducible curves, in particular the curves span the extremal rays, are rational curves with $K_J\cdot [C]=-2$.

If $a>0$ and $h>1$, then we have $b>0$. Hence $C^2=a^2+2ab\ge 1+2a=1+2[C]\cdot T$. Hence $\dim_{SW}([C]-T)=C^2-K_J\cdot [C]-2-2[C]\cdot T\ge 0$. Moreover $SW(K_J-([C]-T))= 0$, since $(K_J-([C]-T))\cdot T=-2-a<0$. Then the wall crossing formula implies $$|SW([C]-T)|=|SW(K_J-([C]-T))-SW([C]-T)|=|1+a|^h\ne 0.$$
And $[C]=([C]-T)+T$ is not an extremal ray. 

If $a>0$ and $h=1$, then we have $$C^2=a^2+2ab>0, \quad -K_J\cdot [C]=2b+a>0.$$ Looking at the classes $l[C]-T$, we have 

\begin{equation*} \label{SWdim3}
\begin{array}{lll}
 \dim_{SW}(l[C]-T)&=&(l[C]-T)^2-K_J\cdot (l[C]-T)\\
 &&\\
 &=&(l^2C^2-lK_J\cdot [C]) -2-2l [C]\cdot T\\
 &&\\
 &=&(l^2a+l)(a+2b)-2-2la\\
 &&\\
 &\ge& l^2a+l-2la-2\\
 \end{array}
 \end{equation*}
 
 It is greater than $0$ if $l>2a$ is large enough. Moreover $SW(K_J-(l[C]-T))= 0$, since $(K_J-(l[C]-T))\cdot T=-2-la<0$. Then the wall crossing formula shows
$$|SW(l[C]-T)|=|SW(K_J-(l[C]-T))-SW(l[C]-T)|=|1+la|^h\ne 0.$$
Hence $[C]=\frac{1}{l}((l[C]-T)+T)$ is not extremal.

\smallskip

Note our argument shows more specifically that any extremal ray paring negatively with the canonical class $K_J$ must be generated by either $E_i$, $T-E_i$ or in the minimal case $T$.

\bigskip

\noindent {\bf Case 2: Rational surfaces:}\\

Recall $C$ is an irreducible curve with $K_J\cdot [C]<0$, which would imply $SW([C])\ne 0$ by Lemma \ref{SW}.  Suppose $C$ is not a rational curve with self-intersection $0$ or $-1$. Since $SW([C])\ne 0$ and $K_J\cdot [C]<0$, by Lemma \ref{int>0}, we know $[C]$ is in the closure of the cone $$P_{K_J}:=\{e\in H^2(M; \mathbb R)|e^2>0, e\cdot E>0 \hbox{ for any }  E\in \mathcal E_{K_J}, e\cdot (-K_J)> 0\}$$where $\mathcal E_{K_J}$ is the set of $-1$ symplectic rational curves. More precisely, it is on the intersection of faces determined by a set of $-1$ classes. Let $S$ be the set of homology classes which are represented by smoothly embedded spheres. We define $$S_{K_J}^+=\{e\in S|g_J(e)=0, e^2>0\}.$$ 
Using this notation, $$\mathcal E_{K_J}=\{e\in S| g_J(e)=0, e^2=-1\}.$$
By Proposition 5.20 in \cite{LZ-generic}, $P_{K_J}=\mathcal S_{K_J}^+$ where the latter is the open cone spanned by $S_{K_J}^+$ (and furthermore equals to the almost K\"ahler cone when $J$ is good generic).  By Lemma 5.24 (2) of \cite{LZ-generic}, each face $F_{E_k}$ of $P_{M_k, K}$ corresponding to $E_k$ is naturally identified with $P_{M_{k-1}, K}$. Here $M_k=\mathbb CP^2\# k\overline{\mathbb CP^2}$ for $k\ne 1$, and $M_1$ might be $S^2\times S^2$ or $\mathbb CP^2\# \overline{\mathbb CP^2}$  depending on whether the set of classes in $\mathcal E_{M_k, K}$ orthogonal to $E_k$ is empty or not. Since the conclusion for $S^2\times S^2$ or $\mathbb CP^2\# \overline{\mathbb CP^2}$ was shown in \cite{LZ-generic}, by induction, we know $[C]=\sum a_i [C_i]$ where $a_i>0$ and $[C_i]\in S_{K_J}^+$. Notice we can choose $a_i$ to be rational numbers since all the classes here are rational (in fact, integral). Furthermore, as noted in \cite{LBL}, 
any  class is in $S_{K_J}^{+}$ is Cremona equivalent to one of the following classes
\begin{enumerate}
 \item $H$, $2H$, 
 \item $(n+1)H-nE_1, n\ge 1$,
 \item $(n+1)H-nE_1-E_2, n\ge 1$.
\end{enumerate}
Here Cremona equivalence refers to the equivalence under the group of diffeomorphisms preserving the canonical class $K_J$.
It is easy to check that each of them could be written as sum of classes of rational curves with square $0$ or $1$.  This implies extremal rays (with $K_J\cdot [C]<0$) have to be spanned by rational curves with $-3\le K_J\cdot [C]<0$. There are countably many such classes, since there are countable many $-1$ curve classes. This finishes the proof of the first statement of our cone theorem on extremal rays.

For the finiteness statement, it makes non-trivial sense only when $M=\mathbb CP^2\# k\overline{\mathbb CP^2}$ with $k\ge 9$. This is because there are finitely many irreducible curves with $K_J\cdot [C]<0$. If $M$ is not rational or ruled, they are $-1$ rational curves $E_i$ where $M$. If $M$ is ruled they are $E_i$, $T-E_i$ or $T$ if it is minimal. If $M=\mathbb CP^2\# k\overline{\mathbb CP^2}$ with $k<9$, they are $-1$ rational curves (there are finitely many when $k<9$, a possibly well-known fact which is also shown later in Proposition \ref{dA}), $H-E$ (when $k=1$) and $H$ (when $k=0$). 

Now we are about to show that when $k\ge 9$ there are only finitely many $-1$ rational curve classes $E$ with bounded symplectic energy $[\omega]\cdot E<\frac{1}{\epsilon}$. Since being symplectic is an open condition, $[\omega]-\delta H$ is still a class of symplectic form when $\delta>0$ is small. Moreover $E$ is always represented by a (possibly reducible) symplectic surface since $SW(E)\ne 0$. Hence $[\omega]\cdot E<\frac{1}{\epsilon}$ would imply $$H\cdot E\le\frac{1}{\delta}([\omega]-\delta H)\cdot E+ H\cdot E=\frac{1}{\delta}[\omega]\cdot E<\frac{1}{\epsilon\delta}.$$ On the other hand there are only finitely many classes $E=aH-\sum c_iE_i$ with $E^2=-1$ and $a=H\cdot E>0$ is bounded from above. Especially, the finiteness statement implies each extremal ray $\mathbb R^+[L_i]$ is not a limit of a sequence of other extremal rays. If this is not true, we have an infinite sequence of $-1$ rational curve classes $C_i$ approaches to a $-1$ rational curve class $E$. Hence, we have $\lim_{i\rightarrow \infty}[\omega]\cdot C_i=[\omega]\cdot E$. In particular, there would be infinitely many $C_i$ satisfying $[\omega]\cdot C_i\le 2[\omega]\cdot E$. This contradicts to our finiteness result above.

To conclude our proof, we first see that $\overline{A}_J^{K_J\ge 0}(M)+\sum \mathbb R^+[L_i]\subset \overline{A}_J(M)$ is a closed convex cone. This is because the above discussion implies $\mathbb R^+[L_i]$ can only have accumulate points in $\overline{A}_J^{K_J\ge 0}(M)$. On the other hand, to show $\overline{A}_J(M)\subset \overline{A}_J^{K_J\ge 0}(M)+\sum \mathbb R^+[L_i]$, we only need to show the inclusion for all classes $e\in \overline{A}_J(M)$ with $K\cdot e<0$. This is exactly what we have proved.
\end{proof}

\begin{remark}
What we have proved is actually a slightly stronger version of cone theorem:
$$\overline{A}^{\mathbb Q}_J(M)=\overline{A}_J^{K_J\ge 0}(M)\cap H^2(M, \mathbb Q)+\sum \mathbb Q^+[L_i],$$where $\overline{A}^{\mathbb Q}_J(M)=\overline{A}_J(M)\cap H^2(M, \mathbb Q)$. 
\end{remark}

The next example shows that it is not true that any Seiberg-Witten nontrivial class of nonzero $J$-genus is an integral combination of curve  classes.
\begin{example}\label{-K8blowup}
Let $M=\mathbb CP^2\#8\overline{\mathbb CP^2}$. Then $SW(-K)=1$ because of $SW(2K)=0$ and the wall crossing formula. However, $-K$ cannot be written as $m_1[C_1]+m_2[C_2]$ with $m_1m_2\ne 0$ and $m_i\in \mathbb Z$ such that $SW([C_i])\neq 0$ for $i=1,2$. This is because the Seiberg-Witten invariant is a deformation invariant. Especially, $[C_1]$ and $[C_2]$ are classes of (possibly reducible) symplectic surfaces in a Del Pezzo surface, {\it i.e.} a $4$-manifold where $-K$ is the class of the symplectic form. Hence $(-K)\cdot [C_i]\ge 1$ since all the classes are integral. It contradicts to $(-K)\cdot (m_1[C_1]+m_2[C_2])=(-K)^2=1$. 

On the other hand, it has the following decomposition with rational coefficients: $$-K=\frac{1}{2}(6H-3E_1-2E_2-\cdots-2E_8)+\frac{1}{2}E_1.$$
\end{example}

One should compare the above example with Proposition \ref{k<8dec}, which shows this kind of example does not exist on $\mathbb CP^2\#k\overline{\mathbb CP^2}, k\le 7$.

We have the following more specific description of the extremal rays.
\begin{prop}\label{exray}
Let $(M, J)$ be a tamed almost complex $4$-manifold. An irreducible curve $C$ is an extremal rational curve as in Theorem \ref{conethm} if and only if 

\begin{enumerate}
\item $C$ is a $-1$ rational curve;
\item $M$ is a minimal ruled surface or $\mathbb CP^2\# \overline{\mathbb CP^2}$, and $C$ is a fiber;
\item $M=\mathbb CP^2$ and $C$ is a projective line.
\end{enumerate}
\end{prop}
\begin{proof}
In the above proof of Proposition \ref{conethm}, we see that when $M$ is not rational or ruled, all the extremal rays for non-minimal manifolds are $-1$ classes. For irrational ruled surfaces, our proof shows that any extremal curve class $C$ has $-2\le K\cdot C<0$. The $K\cdot C=-2$ case is when $C$ is the fiber class $T$. However when $M$ is non-minimal, $T=E+(T-E)$ is not extremal. 

Similarly for rational surfaces, we know that any square $1$ (resp. $0$) sphere class is Cremona equivalent to $H$ (resp. $H-E$). When $M=\mathbb CP^2\# k\overline{\mathbb CP^2}$, $k\ge 2$, we know both classes could be decomposed into two classes with non-trivial Seiberg-Witten invariant: $H=E+(H-E)$, $H-E=E'+(H-E-E')$. Hence they are not extremal. 

When $M=\mathbb CP^2\# \overline{\mathbb CP^2}$, $H=E+(H-E)$  is the sum of two classes with non-trivial Seiberg-Witten invariant. But it is possible that $H-E$ is the only extremal ray. In this case, the effective class $E$ degenerates as $n(H-E)+((n+1)E-nH)$. Then $(n+1)E-nH$ is the class of a $-(2n+1)$ section for the ruled surface. It corresponds to Hirzebruch surfaces $\mathbb F_{2n+1}$ when $J$ is complex. 
\end{proof}

The following lemma gives our information on how a general extremal ray of $A_J(M)$ could be.
\begin{lemma}\label{-extremal}
If $C$ is an irreducible curve with $C^2<0$, then $[C]$ spanned an extremal ray of the curve cone $A_J(M)$. 
\end{lemma}
\begin{proof}
If $\mathbb R^+[C]$ is not extremal, then $[C]=\sum a_i [C_i]$ where $a_i>0$ and $C_i$ are irreducible curves whose homology classes are not on $\mathbb R^+[C]$. For those $C_i$, $[C]\cdot [C_i]\ge 0$. Hence we have the following contradiction
$$0\le \sum a_i[C]\cdot [C_i]=[C]\cdot [C]<0.$$
\end{proof}

The following rationality theorem is originally used in algebraic geometry to prove the cone theorem. It is well known that the statement of cone theorem implies the statement of the rationality theorem. 
\begin{prop}\label{ratthm}
Let $(M, J)$ be a tamed almost complex $4$-manifold such that there is a curve with which $K_J$ pairs negatively.  Let $\omega$ be an almost K\"ahler form on $(M, J)$ with $[\omega]\in H^2(M, \mathbb Q)$. We call a class is $J$-nef if it pairs non-negatively with the curve cone $A_{J}(M)$. Define the nef threshold of $[\omega]$ by $$t_0=t([\omega])=\sup\{t\in \mathbb R: tK_J+[\omega] \hbox{ is $J$-nef}\}.$$ Then the nef threshold is a rational number.   
\end{prop}
\begin{proof}
It is easy to see that $$t_0=\sup \frac{L_i\cdot [\omega]}{-K_J\cdot L_i}.$$ If there are finitely many extremal curves $L_i$, then $t_0$ is a rational number by this formula. In general, since $K_J$ is not $J$-nef, there exists a small number $\epsilon$ such that $\frac{1}{\epsilon}K+[\omega]$ is not $J$-nef which pairs negatively  with only finitely many extremal curves. Hence $t_0$ is the supremum of $ \frac{L_i\cdot [\omega]}{-K_J\cdot L_i}$ for these finitely many $L_i$, which has to be a rational number. In fact, in this case, $t_0$ is an integer since the situation of the second statement of Theorem \ref{conethm} happens only when all $L_i$ are $-1$ rational curves.

Moreover, when $[\omega]\in H^2(M, \mathbb Z)$, then the denominator of $t_0$ is no greater than $3$. It is not an integer only when the last two cases of Proposition \ref{exray} happen.
\end{proof}

\begin{cor}\label{-1}
Let $M=N\#\overline{\mathbb CP^2}$ be a non-minimal symplectic $4$-manifold which is not diffeomorphic to $\mathbb CP^2\# \overline{\mathbb CP^2}$. Then for any tamed $J$ on $M$, there exists at least one smooth $J$-holomorphic $-1$ rational curve.

More precisely,  for any exceptional class $E$, we have a decomposition $E=\sum a_i[C_i]$ with $0<a_i\in \mathbb Q$ and $C_i$ irreducible curves, such that there is at least one $-1$ rational curves $C_i$. 


\end{cor}
\begin{proof}
By the assumption, there is at least one $-1$ rational curve class $E\in H_2(M, \mathbb Z)$. If there is an irreducible $J$-holomorphic subvariety in class $E$, then we are done since it will be smooth by adjunction inequality. If not, since the curve cone is a convex cone, the class is written as $E=\sum a_i[C_i]$ where all $C_i$ are irreducible $J$-holomorphic subvarieties and all $[C_i]$ are extremal curve classes. Furthermore, $a_i\in \mathbb Q$ since all the classes $E$ and $[C_i]$ are in $H_2(M, \mathbb Z)$.

Since $K\cdot E=-1<0$, we know there is at least one $C_i$ (say $C_1$) such that $K\cdot [C_1]<0$. By the cone theorem, this irreducible $C_1$ has to be a rational curve with $-3\le K\cdot [C_1]<0$. 
Moreover, by comparing to the list in Proposition \ref{exray}, $C_1$ has to be the class of a $-1$ rational curve.
\end{proof}



Notice that the statement that $\mathbb CP^2\# k\overline{\mathbb CP^2}$ has at least one smooth $-1$ rational curve is first proved in \cite{pinn}. Actually, it shows that a class with minimal symplectic energy is such a  class. However, our proof gives more precise result. 
The second statement of the above corollary is crucial for our later applications.

It is interesting to compare our picture here for a general tamed almost complex structure to the bend-and-break in algebraic geometry. The bend-and-break technique in algebraic geometry starts with an irreducible curve $C'$ with $K\cdot [C']<0$. One chooses a normalization $f: C\rightarrow M$ of $C'$. Then it contains two parts. The first, the ``bend" part, is to compose the normalization with automorphisms of $C$, possibly in characteristic $p$ when $g(C)>1$, such that $-K\cdot f'(C)-g(C)\dim_{\mathbb C} M>0$. This would guarantee one could deform curves in a class which is a multiple of $[C']$. The second, the ``break" part, shows that this family must degenerate to $f''(C)+$(sum of rational curves).

Our argument is sort of a reverse process. We show that all the extremal rays with negative $K$ pairing are spanned by rational curves. And thus a higher multiple of  the curve class $C$ with $K\cdot C<0$  will degenerate to a reducible curve with at least one extremal ray as one of its irreducible components.

In general, it is not true that we always have a reducible curve in class $[C]$ if $C$ is an irreducible $J$-holomorphic curve of positive genus such that $K\cdot [C]<0$ as we have seen in Example \ref{-K8blowup}. Here is an example for ruled surfaces.

\begin{example}\label{ruledex}
Let $M$ be the non-trivial $S^2$ bundle over $\Sigma_h$. Let $U$ be the class of a section with $U^2=1$. For $a=\lceil \frac{h-2}{2}\rceil$, we have $SW(U+aT)\ne 0$ and $SW(U+(a-1)T)=0$. Thus for a generic tamed almost complex structure, we do not have reducible curves in class $U+aT$. However, as shown in the proof of Theorem \ref{conethm}, we do have curves in class $lU-T$ when $l>2$. Hence there is always a reducible subvariety in class $l(U+aT)$. 
\end{example}

On the other hand, as shown in our proof, the class $[C]$ itself contains reducible curves in many circumstances. We also have the following slight variant of the well known fact that rational or ruled $4$-manifolds are symplectically uniruled.

\begin{prop}\label{ample}
Let $(M, J)$ be a tamed almost complex $4$-manifold with a $J$-ample anti-canonical class $-K$ ({\it i.e.} $-K$ pairs positively with all curves). Then $M$ contains a rational curve. In fact, through every point of $M$ there is a rational curve $C$ such that $$0<-K\cdot [C]\le 3.$$
\end{prop}
\begin{proof}
By Taubes' theorem \cite{T}, when $-K$ is ample, $M$ has to be rational or ruled. Furthermore, it cannot be irrational ruled, otherwise $(-K)\cdot (U+aT)\le 0$ and $SW(U+aT)\ne 0$ where $a=\lceil \frac{h-2}{2}\rceil$ when $M$ is the non-trivial $S^2$ bundle over $\Sigma_h$ and $a=\lceil \frac{h-1}{2}\rceil$ when $M=S^2\times \Sigma_h\# k\overline{\mathbb CP^2}$. Hence, $M$ is rational and we choose homology basis such that $K=-3H+E_1+\cdots +E_k$.

Since $K\cdot [C]<0$ for all irreducible curves, then either the curve is a $-1$ rational curve or it has $(K-[C]) \cdot [D]<0$ for any irreducible curve (and then any subvariety) $D$. In both cases, $\dim_{SW}([C])=C^2-K\cdot [C]=2g-2-2K\cdot [C]\ge 0$.  Hence $SW([C])\ne 0$ by Lemma \ref{int>0} and the wall-crossing  formula.

Since $SW([C])\ne 0$ and $H$ is represented by an irreducible $J'$-holomorphic curve with positive self-intersection for a generic $J'$, by Lemma \ref{int>0}, $H\cdot [C]\ge 0$ for all curves $C$.
In other words, $H$ is $J$-nef when $-K$ is $J$-ample. Then by Theorem 1.5 of \cite{LZrc}, we know any irreducible component $C_i$ of a subvariety in class $H$ is a rational curve with $0<-K\cdot [C_i]<-K\cdot H=3$. Then the conclusion follows since there is a subvariety in class $H$ passing through any given point.
\end{proof}

\section{$\mathbb CP^2\#2\overline{\mathbb CP^2}$}\label{k=2}

This section serves as a link between Sections 2 and 4. We first recall some general results for the homology classes of irreducible subvarieties on $\mathbb CP^2\#k\overline{\mathbb CP^2}$. Then we apply them to give an explicit description of the negative curves on $\mathbb CP^2\#2\overline{\mathbb CP^2}$. In particular, we show there are at least two smooth $-1$ rational curves for any tamed almost complex structure on it in Theorem \ref{cp2+2}.  A full description of the curve cone is given in Theorem \ref{cp2+2cone}. 
 The information on the curve cone helps us to obtain a Nakai-Moishezon type duality, Theorem \ref{NM2}. 

\subsection{The curve cone and the $K-$symplectic cone}
The $K-$symplectic cone for a class $K\in H^2(M;\mathbb Z)$ introduced in \cite{LL}:
\begin{equation}\label{Kcone}
\mathcal C_{M, K}=\{e\in H^2(M;\mathbb R)|e=[\omega] \hbox{ for some $\omega$ with $K_{\omega}=K$}\}.
\end{equation}
Here $K_{\omega}$ is the symplectic canonical class of $\omega$. Suppose $\mathcal C_{M,K}$ is non-trivial, $[\omega]\in \mathcal C_{M,K}$ and $b^+(M)=1$, by Theorem 3 in \cite{LL},
\begin{equation} \label{cmk}\mathcal C_{M,K}=\{e\in \mathcal{FP}(K)|e\cdot E>0, E \in \mathcal E_{K}
\}.\end{equation}
Here $\mathcal{FP}(K)$ is the connected component of $\mathcal{P}=\{e\in H^2(M, \mathbb Z)|
e^2>0 \}$ containing $[\omega]$.  Notice by the light cone lemma, both $\mathcal{FP}(K)$ and $\mathcal
C_{M,K}$ are convex cones. Recall the following statement which is called the light cone lemma in the literature, which is in the guise of Cauchy-Schwartz inequality. The cone of elements with positive squares have two components. The forward cone means one of the two connected components containing a given element with positive square. In our applications, it is usually a class of symplectic form. 

\begin{lemma}[light cone lemma]\label{light}
For the light cone of signature $(1, n)$ ($n\ne 0$), any two elements in the forward cone have non-negative dot product. Especially, if the dot product is zero then the two elements are proportional to each other.
\end{lemma}

We will state a structural description of the $K-$symplectic cone of rational surfaces, which might be known for experts. Without loss, we suppose $K=-3H+\sum_iE_i$.

\begin{prop}\label{corner01}
Let $M=\mathbb CP^2\#k\overline{\mathbb CP^2}$. 
\begin{enumerate}
\item When $k<9$,  the $K-$symplectic cone is a cone over a polytope whose corners are the classes of the symplectic spheres with canonical class $K$ and self-intersection $0$ or $1$. 

\item When $k\ge 9$, all the extremal rays $\mathbb R^+e$ of the $K-$symplectic cone having $K\cdot e<0$ are generated by the classes of the symplectic spheres with canonical class $K$ and self-intersection $0$ or $1$.
\end{enumerate}
\end{prop}
\begin{proof}
We first assume $k<9$. The $K-$symplectic cone is a polytope follows from Proposition 2.7(1) in \cite{FM}. In fact, the $K-$symplectic cone is a $P$-cell when $k<9$, and a $P$-cell has no round boundary. One could also see Section 5.4 of \cite{LZ-generic} for an overview of the results in \cite{FM} using notations similar to this paper's. For a more direct argument, see \cite{McSc}.

We now show that the corners of the polytope, which correspond to the extremal rays of the $K-$symplectic cone, are those classes which can be represented by the symplectic spheres with canonical class $K$ and self-intersection $0$ or $1$. This statement follows from the light cone lemma. The corners of the polytope are the intersection of several hyperplanes determined in our situation by classes in $\mathcal E_K$. Recall the $K-$symplectic cone is a cone over a polytope implies that the corners of the polytope are in $\overline{\mathcal P}$. 

There are two types of corners. If the classes determining the hyperplanes around the corner ray $\mathbb R^+e$ are orthogonal to each others, without loss we can assume them to be $E_1, \cdots, E_l$. We know $l=k$ otherwise the intersection would include any classes $aH-bE_k$ with $a\ge b>0$, hence is not a ray. But when $l=k$, the primitive class corresponding to the corner is a class Cremona equivalent to $H$, {\it i.e.} the class of a symplectic sphere with self-intersection $1$. 

If two of the classes determining the hyperplanes around a corner are not orthogonal, say $E$ and $E'$ (both in $\mathcal E_K$), then the corner class $e$ is also orthogonal to $E+E'$ which is of non-negative self-intersection. By the light cone lemma, $e$ is proportional to $E+E'$ and $(E+E')^2=0$. The equality happens only if $E\cdot E'=1$ and $E+E'$ is the class of a symplectic sphere of self-intersection $0$ (since $K\cdot (E+E')=-2$). If we choose $e$ to be primitive, then $e=E+E'$.

When $k\ge 9$, by Proposition 2.7 of \cite{FM}, the intersection $\mathcal C_{M, K}\cap \{e\in H^2(M; \mathbb R)| K\cdot e\le 0\}$ is a $P$-cell which has no round boundary. Hence, the extremal rays with $K\cdot e< 0$ corresponding to intersections of the ordinary walls ({\it i.e.} the hyperplanes determined by classes in $\mathcal E_K$ rather than the canonical class $K$). Then we have exactly the same two types of corners, and the same argument applies.
\end{proof}

\begin{remark}\label{k>9round}
The $K-$symplectic cone has round boundary when $k>9$.
For examples, one can take the anti-canonical class $e=-K$. The class $e\notin \mathcal C_{M, K}$ when $k>9$ since $e^2<0$ but $e\cdot E=1>0$ for all $E\in \mathcal E_K$. 

Moreover, one can indeed get an open set of round boundary. Since $\mathcal C_{M, K}\ne \emptyset$, we choose an open set $B$ of it. 
A class $-K+a_{\omega}[\omega]$, where $a_{\omega}>0$ and $[\omega]\in B$, is on the round boundary of $\mathcal C_{M, K}$ if  $a_{\omega}$ solves $(-K+a_{\omega}[\omega])^2=0$ (since $(-K+a_{\omega}[\omega])\cdot E>0$ for all $E\in \mathcal E_K$). This is a quadratic equation and such $a_{\omega}$ exists since $K^2\cdot [\omega]^2<0$. 

For a detailed discussion on the round boundary, see the new edition of \cite{McS}. 
\end{remark}

The $K-$symplectic cone can be used to restrict the classes of negative square in the curve cone. The following lemma is simple but also very useful.

\begin{lemma} \label{curvecone}
Suppose $b^+(M)=1$. A cohomology class
is in the curve cone only if it is positive on some extremal ray of the $K-$symplectic cone.
\end{lemma}

\begin{proof}
Assume the class $e$ is in the curve cone $A_J(M)$ for some tamed almost complex structure $J$ with $K_J=K$. Let $\omega$ be a symplectic form taming $J$. Then $e\cdot [\omega]>0$. Since $\mathcal C_{M, K}$ is convex, there is an extremal ray pairing positively with $e$.
\end{proof}

The next lemma is on the constraints of the curve classes
provided by the adjunction inequality.

\begin{lemma}\label{B}
Suppose a class $B=\alpha H+\sum \beta_i E_i$  has
an irreducible curve representative. 

If $\alpha>0$ then   $ |\beta_i|\leq |\alpha|$ for each $i$, and $|\alpha|=|\beta_i|$ only when $\alpha=-\beta_i=1$.

If $\alpha=0$, then $B=E_i-\sum_jE_{k_j}$.

If $\alpha< 0$, then $ |\beta_i|\leq |\alpha|+1$, and $|\beta_i|=|\alpha|+1$  only if $\beta_i=-\alpha+1$.
\end{lemma}
\begin{proof}
By adjunction formula, and the fact that $\gamma^2+\gamma\geq 0$ for any integer $\gamma$,  we have
$$(\alpha-1)(\alpha-2)\geq \beta_i(\beta_i+1).$$

Then all the conclusions are clear when $\alpha\neq 0$. 

For $\alpha=0$, we first get that $\beta_i=0$, $-2$ or $\pm 1$. However when $\alpha=0$, some $\beta_i$ is positive. Otherwise, $J$ will not be tamed by Equation \eqref{cmk}. Then $B=E_i-\sum_jE_{k_j}$ holds since $\sum \beta_i(\beta_i+1)\le 2$ by  adjunction formula. 
\end{proof}

\subsection{Curves on $\mathbb CP^2\#2\overline{\mathbb CP^2}$}

We have shown that there is at least one smooth $J$-holomorphic $-1$ rational curve for any tamed $J$ on non-minimal symplectic manifold except for $\mathbb CP^2\#\overline{\mathbb CP^2}$ in Corollary \ref{-1}. Now, we will show that when $M=\mathbb CP^2\#2\overline{\mathbb CP^2}$, we actually have at least two $-1$ sphere classes. Notice it is not true when $M=\mathbb CP^2\#k\overline{\mathbb CP^2}$ for $k>2$.

\begin{theorem}\label{cp2+2}
There are  at least two smooth $-1$ $J$-holomorphic rational curves for any tamed $J$ on $\mathbb CP^2\#2\overline{\mathbb CP^2}$.
\end{theorem}

\begin{proof}
First, there is at least one $-1$ smooth rational curve by Corollary \ref{-1}. We first assume the class $E_2$ has such a smooth representative.

\begin{lemma}\label{A} Suppose $M=\mathbb CP^2\#2\overline{\mathbb CP^2}$ and the class $E_2$ is the class of a smooth $-1$ rational curve.
If a class $A=aH+bE_1+cE_2\ne E_2$ with 
$a\le 0$ has an irreducible curve representative, then

(i) $b>0,  a=1-b$ and $c=0$ or $-1$.

(ii)  $A$ is a sphere class.

(iii) $A\cdot A<0$.

(iv) $A$ is the only such class.
\end{lemma}
\begin{proof}
First $c\leq 0$ since $A\cdot E_2=-c\geq 0$.

Second, the set $\mathcal E_K=\{E_1, E_2, H-E_1-E_2\}$. It could be checked directly by adjunction formula, or from section 4 (in particular, Lemma \ref{a<0} and Proposition \ref{a>0}). Then by \eqref{cmk}, the extremal rays of the $K-$symplectic cone are spanned by
$$H, \quad H-E_1, \quad H-E_2.$$
 As $a\le 0$ and $c\le 0$, we have
$A\cdot H\le 0$ and $A\cdot (H-E_2)\le 0$. Therefore by Lemma \ref{curvecone}, $A\cdot (H-E_1)$ is strictly positive. This means that
$a+b>0$, i.e.  $b>-a\ge 0$.

By the adjunction formula
\begin{equation}
(a-1)(a-2)\geq b(b+1)+c(c+1).
\end{equation}
The only possibility is as claimed in $(i)$, $A= (1-b)H+ bE_1$ or $A=(1-b)H+bE_1-E_2$ with $b\ge 1$. Items $(ii)$ and $(iii)$ are then direct to check.

For $(iv)$, suppose $A'=(1-b')H+b'E_1 $ or $(1-b')H+b'E_1-E_2$ is another such class. Then $A\cdot A'=1-(b+b')$ or $-(b+b')$ is negative.
\end{proof}

\begin{lemma}\label{cca<0}
Suppose $M=\mathbb CP^2\#2\overline{\mathbb CP^2}$  and the class $E_2$ is the class of a smooth $-1$ rational curve. If $A=(1-s)H+sE_1$ or $(1-s)H+sE_1-E_2$ with $s\ge 1$ is in the curve cone, then a class $D=H+vE_1+wE_2$  is in the curve
cone only if $v> -1$ or $v=-1, w\ge -1$.
\end{lemma}

\begin{proof}
Since $s\ge 1$,  $D$ must be of the form $pA+\sum B_i$, where $p\ge 0$ and  $B_i=\alpha_iH+\beta_iE_1+\gamma_iE_2$ denote a positive multiple of homology classes of irreducible curves which are not on the ray through $A$. By Lemma \ref{A} $(iv)$, we have $\alpha_i>0$ or $B_i=E_2$. Then pairing with $H$, we have
$$1=p(1-s)+\sum \alpha_i.$$ Now pairing with $E_1$, we have
$$v=ps+\sum \beta _i=p-1+ \sum (\alpha _i+\beta_i)\geq -1$$ by Lemmas \ref{B} and \ref{A}. Moreover, $v=-1$ only if $p=0$, and by Lemma \ref{B}, we have $$w=\sum \gamma_i \ge -\sum \alpha_i=-1.$$
\end{proof}

\begin{cor}\label{1to2}
Suppose $E_2$ is the class of a smooth $-1$ rational curve. Then so is either the class $H-E_1-E_2$ or the class $E_1$. And in the latter case,
both $E_1$ and $H-E_1-E_2$ are the classes of smooth $-1$ rational curves.
\end{cor}
\begin{proof}
Recall $\mathcal E_K=\{E_1, E_2, H-E_1-E_2\}$. 
Suppose $E_2$ has an embedded representative and $E_1$ does not.
Note the class $H-E_1-E_2$ is in the curve cone since $SW(H-E_1-E_2)\ne 0$.
By Corollary \ref{-1}, there is a  $-1$ rational curve $C_i$ in a decomposition $H-E_1-E_2=\sum_ia_i[C_i]$ (in other words, an irreducible component of a $\mathbb R$-variety in class $H-E_1-E_2$).
By our assumption, this class cannot be $E_1$.
If this class is $H-E_1-E_2$, we are done.

If this class is $E_2$, then $H-E_1-lE_2$, with $l>1$, is in the curve cone. If we have an irreducible curve in class $(1-s)H+sE_1$ or $(1-s)H+sE_1-E_2$ with $s\ge 1$, it will contradict to the Lemma \ref{cca<0}. So all irreducible curves other than $E_2$ have $a=C\cdot H> 0$. By Lemma \ref{B}, all such irreducible curves $aH-b_1E_1-b_2E_2$ have $b_2\le a$. Hence $H-E_1-lE_2$ can never be in the curve cone for $l>1$, contradicting the assumption that $E_2$ appears in the decomposition of $H-E_1-E_2$.

  If both $E_1$ and $E_2$ have embedded representatives, same argument shows that neither can appear in the decomposition of $H-E_1-E_2$.
\end{proof}

Now, to finish the proof of Theorem \ref{cp2+2}, we are left with case that the class $H-E_1-E_2$ has an embedded representative.

By Corollary \ref{-1}, there will be a $-1$ rational curve as an irreducible component of the $\mathbb R$-subvariety representing class $E_1$ or $E_2$. Suppose that there is no irreducible curve with non-positive $H$ coefficient. Then
 $H-E_1-E_2$ cannot appear as the class of such $-1$ rational curve. Thus
 the $-1$ class component in $E_1$ is either $E_1$ or $E_2$. Hence in this situation,
 there are at least two $-1$ rational curves.

Thus we assume that there is an irreducible curve in  class
 $A=aH+bE_1+cE_2$  with $a\le 0$.

Since $a\le 0$,  either $a+b>0$ or $a+c>0$ by Lemma \ref{curvecone}.

Without loss, we assume that  $a+b>0$. We will show that in this case  $E_1$ or $E_2$ must have an embedded representative.

First by adjunction formula  $$(a-1)(a-2)\ge b(b+1)+(c^2+c).$$
On the other hand $b\ge -a+1>0$. Hence the only possibility for the adjunction holds would be $a=-b+1$, $c=0$ or $-1$ and $g=0$. Then $A= (1-b)H+ bE_1$ or $A=(1-b)H+bE_1-E_2$ with $b\ge 1$.

If $E_1$ and $E_2$ do not have irreducible representative, then $H-E_1-E_2$ is the only class of extremal irreducible curve with $K\cdot C<0$ as shown in Corollary \ref{-1}. We now look at irreducible curves with $K\cdot C\ge 0$. First let $[C]=aH-b_1E_1-b_2E_2, a>0$. Then $K\cdot C\ge 0$ implies $$0<3a\le b_1+b_2.$$But by local positivity of intersections, $C\cdot (H-E_1-E_2)\ge 0$, which is $a\ge b_1+b_2$. It is a contradiction. Hence $a\le 0$ and the curves classes are calculated as above. Hence $E_i$ will be a linear combination of these classes and $H-E_1-E_2$. If we write a class as $aH+b_1E_1+b_2E_2$, then all the above classes will contribute non-positively to $2a+b_1+b_2$. However $E_i$ has positive $2a+b_1+b_2$. This is a contradiction. Hence there is an irreducible curve in class $E_1$ or $E_2$.
\end{proof}

The above discussion actually gives the following description of the curve cone. By Theorem \ref{cp2+2}, there is always an irreducible $J$-holomorphic curve in class $E_1$ or $E_2$. Hence, without loss, we could assume $E_2$ has an irreducible representative.

\begin{theorem}\label{cp2+2cone}
Let $J$ be a tamed almost complex structure on $\mathbb CP^2 \# 2\overline{\mathbb CP^2}$ such that there is a smooth $J$-holomorphic curve in the class $E_2$. Then the curve cone $A_J(M)$ is generated by $3$ classes. They are either
$$\alpha=(1-s)H+sE_1, \quad  \beta=E_2, \quad \gamma=H-E_1-E_2,$$
or
$$ \alpha=(1-s)H+sE_1-E_2, \quad  \beta=E_2, \quad \gamma=H-E_1-E_2,$$ where $s\ge 1$.
\end{theorem}
\begin{proof}
First by Lemma \ref{1to2}, there is always an irreducible curve in class $H-E_1-E_2$. By the argument in the last paragraph of the proof of Theorem \ref{cp2+2}, a curve class $C$ with $C\cdot H>0$ is always spanned by $H-E_1-E_2$, one of $E_i$ say $E_2$ and another irreducible curve with non-positive pairing with $H$. When $E_2$ has irreducible representative, the last curve class is $(1-s)H+sE_1$ or $(1-s)H+sE_1-E_2$ with $s\ge 1$. By Lemma \ref{A} (iv), such a curve is unique. This completes our proof.
\end{proof}

Now we can study the $\geq 0-$dual  of the curve cone $A_J(M)$.

If $A_J$ is generated by $E_1, E_2, H-E_1-E_2$, then its $\geq 0-$dual is generated by
$H, H-E_1, H-E_2$. 

Let us assume that $E_2$ is irreducible. We discuss the two cases in Theorem \ref{cp2+2cone}.
In the first case, the $\geq 0-$dual of $A_J$ is generated by $$A=sH-(s-1)E_1, \quad B=H-E_1, \quad C=sH-(s-1)E_1-E_2.$$
In the second case, the $\geq 0-$dual of $A_J$ is generated by $$A=sH-(s-1)E_1, \quad B=H-E_1, \quad C=(s+1)H-sE_1-E_2.$$

In both cases, $A$ and $C$ are in $S_{K_J}^+$ and the ray spanned by $B$ is approximated by the sequence $pH-(p-1)E_1-E_2, p\to \infty$ in $S_{K_J}^+$.

All the above actually shows the following

\begin{prop}\label{s=p}
For any tamed $J$ on $\mathbb CP^2 \# 2\overline{\mathbb CP^2}$, we have $$\mathcal S_J=\mathcal P_J.$$
\end{prop}

Recall that by definition $\mathcal P_J=A_J^{\vee, >0}(M) \cap \mathcal P$ as in the introduction (it should not be confused with $P_{K_J}$). The spherical cone $\mathcal S_J$ here is defined to be the interior of the convex cone generated by big $J$-nef classes ({\it i.e.} $J$-nef classes with positive square) in $S_{K_J}$ if it is of dimension $3$.

\subsection{Nakai-Moishezon type theorem for almost K\"ahler structure on $\mathbb CP^2 \# 2\overline{\mathbb CP^2}$}
With Proposition \ref{s=p} in hand, we can establish the Nakai-Moishezon and Kleiman type theorems  for almost K\"ahler $J$ on $\mathbb CP^2 \# 2\overline{\mathbb CP^2}$. 

\begin{theorem}\label{NM2}
Let $M=\mathbb CP^2\# 2\overline{\mathbb CP^2}$. For any almost K\"ahler $J$, 
the $J$-compatible cone $\mathcal K_J^c(M)$
is dual to the $J$-curve cone $A_J(M)$, i.e.
$$\mathcal K_J^c= \mathcal P_J=A_J^{\vee, >0}(M) .$$
\end{theorem}
\begin{proof}
It is clear that $\mathcal K_J^c\subset \mathcal P_J.$ If $J$ is almost K\"ahler, we have $\mathcal S_J\subset \mathcal K_J^c$ (Lemma 5.18 in \cite{LZ-generic}). By Proposition \ref{s=p}, we have $$\mathcal K_J^c= \mathcal P_J=\mathcal S_J.$$ 
The second equality $\mathcal P_J=A_J^{\vee, >0}(M)$ holds because the classes $A, B, C$ and thus their positive combinations all have non-negative squares.
\end{proof}

We remark that the techniques in \cite{LZ-generic} to construct almost K\"ahler form for a tamed $J$ fail in this situation.

First we need to construct a Taubes current. Here a current is a differential form with distribution coefficients. Hence it represents a real cohomology class when pairing with smooth closed forms in the weak sense. A Taubes current is a closed, positive $J$-invariant current $\Phi$, which satisfies $$k^{-1}t^4\le \Phi(if_{B_t(x)}\sigma\wedge \bar{\sigma})<kt^4.$$ Here $\sigma$ denotes a point-wise unit length section of $T^{1,0}M|_{B_t(x)}$. 
The usual technique for the construction is to integrate certain part of the moduli space of the subvarieties in a $J$-ample class $e$, {\it i.e.} a cohomology class pairing positively with any curve classes, with $g_J(e)=0$.
However in general, as in current situation, we do not have any such classes. We may only have big $J$-nef spherical classes. In this situation,  we are still able to produce a weak version of Taubes current in class $e$: a closed, non-negative $J$-invariant current $\Phi_e$, satisfying $$0\le \Phi_e(if_B\sigma\wedge \bar{\sigma})<kt^4.$$ It will vanish along the vanishing locus $Z(e)$, {\it i.e.} the union of irreducible subvarieties $D_i$ such that $e\cdot D_i=0$. But  over any $4$-dimensional compact submanifold with boundary $K$ of the complement $M(e)=M\backslash Z(e)$, it is a Taubes current with the constant $k>1$ depending only on $K$.

If we have sufficiently many big $J$-nef classes, we could produce genuine Taubes currents by the following Proposition 5.7 in \cite{LZ-generic}.

\begin{prop}\label{form}
Let $e_i$ be big $J$-nef  classes in $S_{K_J}$ and the zero locus of $e_i$ is denoted by $Z_i$. 
If $\cap Z_i =\emptyset$, then there is a Taubes current in the class $e=\sum_i a_ie_i$, with $a_i>0$.
\end{prop}

Finally, we apply the following regularization result of \cite{T2009} (see also \cite{Zmann}) to obtain an almost K\"ahler form in the class $e$.

\begin{theorem}
In a $4$-manifold $M$ with $b^+(M)=1$, if we have a Taubes current $T$, 
then there is an almost K\"ahler form $\omega$, s.t. $[\omega]=[T]$.
\end{theorem}

Hence to construct an almost K\"ahler form by the subvariety-current-form method, we are reduced to prove that there exist big $J$-nef classes $e_i$ in $S_{K_J}$, such that the intersection of the zero locus  $\cap Z(e_i) =\emptyset$. We claim it is impossible if our $J$ is assumed only to be tamed. In the below, $\alpha, \beta, \gamma$ are those classes in Theorem \ref{cp2+2cone}. The main point is there is no class $e\in S_J^+$ such that $Z(e)=\alpha$ by simple homological calculation. Since $\beta\cdot \gamma=1$, we have $\beta\cap \gamma\ne \emptyset$. Because of the above observation, any class $e\in S_J^+$ will have $\beta\cap \gamma \subset Z(e)$. Hence $\beta\cap \gamma \subset \cap Z(e_i)$, which is then not empty.

\section{Configurations of negative curves on rational and ruled surfaces}
Almost complex structures are different from complex structures at blowing up and down. More precisely, when we have an irreducible holomorphic $-1$ sphere, we can always blow it down by Castelnuovo's criterion. However, generally we cannot blow down a smooth $J$-holomorphic $-1$ rational curve for a tamed almost complex structure $J$.

In this section, we study the negative curves in a tamed almost complex rational or ruled $4$-manifolds, which might not be mentioned explicitly in each statement. 

\subsection{Negative curves on $\mathbb CP^2\#k\overline{\mathbb CP^2}$ with $k\le 9$}
In this subsection, we will study the negative curves on rational surfaces. Without loss, we suppose $K=-3H+\sum_iE_i$. We will show that they have to be spheres for $\mathbb CP^2\#k\overline{\mathbb CP^2}$ with $k\le 9$. This will enable us to determine the curve cone and show that the configurations of negative curves for tamed almost complex structures are all realized by complex structures (Theorem \ref{main1}). Note that the case of $S^2\times S^2$ and $\mathbb CP^2\# \overline{\mathbb CP^2}$ have been done in \cite{LZ-generic}, $\mathbb CP^2\#2\overline{\mathbb CP^2}$ is done in section \ref{k=2}.

 We first take a look at the irreducible curve classes $C=aH+\sum b_iE_i$ with $C^2<0$ and $a<0$. 

\begin{lemma}\label{a<0}Let $M=\mathbb CP^2\#k\overline{\mathbb CP^2}$. 
If $C=aH+\sum b_iE_i$ with 
$a<0$ is represented by an irreducible curve, then 
\begin{itemize}

\item $C=-nH+(n+1)E_1-\sum_{k_j\neq 1}E_{k_j}$ up to a Cremona tranformation, {\it i.e.} a diffeomorphism preserving the canonical class. 
\item Or $C=f^*C'$, where $f$ is a Cremona transformation and $C'$ is a class with $a'>0$. 
\end{itemize}

\end{lemma}

\begin{proof}
The proof is eventually similar to that of Lemma \ref{A}. As we suppose our canonical class $K=-3H+\sum E_i$, there are two types of the generators of extremal rays of the $K-$symplectic cone. The first type is the classes $F$ with $F^2=1$ which can be represented as a sphere. By \cite{LBL}, those are Cremona equivalent to $H$.  
The second type of classes are those Cremona equivalent to $H-E_1$.

Let us first suppose that $C$ pairs non-positively with all the classes of the first type. By Lemma \ref{curvecone}, at least one of these classes equivalent to $H-E_1$ pairs positively with $C$. If we suppose it is $H-E_1$, then we know that $b_1>-a>0$. Then by the adjunction formula $(a-1)(a-2)\ge \sum b_i(b_i+1)$. Thus, the only possibility is as claimed,  $C=-nH+(n+1)E_1-\sum_jE_{k_j}$ which has $g_J(C)=0$. 

If $C\cdot F>0$ for some $F$ of the first type, then we can first change the class $F$ to $H$ by a diffeomorphism $f$ preserving the canonical class. The class $C$ changes to $C'$ at the same time  and thus $C'\cdot H>0$. Thus $C'$ is a class with $a'>0$ 
and $C$ is pull-back of it by a Cremona transformation.
\end{proof}

The latter case could happen. For example when $C=-H+E_1+E_2+E_3$, $F=2H-E_1-E_2-E_3$. Then $C$ is equivalent to $H-E_1-E_2-E_3$ after a Dehn twist along $H-E_1-E_2-E_3$. 
However, on $\mathbb CP^2\#k\overline{\mathbb CP^2}$ with $k\le 9$, all $-1$ rational curve classes $C=aH+\sum b_iE_i$ have $a\ge 0$. See Lemma \ref{-1a>0}.

The case when $a=0$ is investigated in Lemma \ref{B}. The only possible curves are $E_i-\sum_{k_j\neq i}E_{k_j}$.

Now, let us take a look at the case of $a>0$. 

\begin{prop} \label{extreme}
Suppose $M=\mathbb CP^2\#k\overline{\mathbb CP^2}$, $k\le 9$. 
\begin{enumerate}
\item Then any irreducible curves $C$ with $C^2<0$ are smooth spheres.
\item If  $k\le 8$, any irreducible curves with $C^2\le 0$ are smooth spheres.
\end{enumerate}
\end{prop}
\begin{proof}
Let $C=aH-\sum b_iE_i$. The case when $a\le 0$ is discussed above. The only undetermined case, the second bullet of Lemma \ref{a<0} is reduced to the case of $a>0$. Hence, we suppose $a \ge1$ below.

Because $C^2<0$, we can suppose 
\begin{equation}\label{squ-}
c^2+a^2=b_1^2+\cdots b_k^2, \quad c\in \mathbb R\backslash \{0\}.
\end{equation}
In other words, $c^2=-C^2$.

If $C$ is an irreducible curve and is not a sphere, then by adjunction formula,
\begin{equation}\label{adj-}
c^2+3a\le b_1+\cdots+b_k.
\end{equation}
Hence, by Cauchy-Schwartz inequality, we have $(c^2+3a)^2\le k(c^2+a^2)\le 9(c^2+a^2)$, {\it i.e.} $$6ac^2+c^4-9c^2\le 0.$$ This inequality is possible only when $a=1$. Then \eqref{adj-} becomes $$b_1+\cdots +b_k\ge 3+c^2.$$ It contradicts to \eqref{squ-}, which reads as $$3+c^2\le b_1+\cdots +b_k\le b_1^2+\cdots +b_k^2= 1+c^2.$$

This contradiction shows that $C$ should be a sphere.

We also notice that when $k\le 8$, any irreducible curves with $C^2\le 0$ are spheres. This is because if $C^2=0$, formulae  \eqref{squ-}, \eqref{adj-} lead to the contradiction $$ka^2\ge (3a)^2.$$
\end{proof}
Notice that the statement is sharp, in the sense that it is no longer true for $k\ge 10$ (resp. $k\ge 9$), since the anti-canonical class $K$ could be the class of an elliptic curve with $K^2<0$ (resp. $K^2\le 0$).

A simple observation is the same argument could be used to prove the following.

\begin{lemma} \label{-1a>0}
Suppose $M=\mathbb CP^2\#k\overline{\mathbb CP^2}$, $k\le 9$. Then all $-1$ rational curve classes $C$ have $[C]\cdot H\ge 0$.
\end{lemma}
\begin{proof}
We write $[C]=-aH+\sum b_iE_i$. We assume $a>0$. We observe the class $-[C]$ satisfies $(-[C])^2=-1<0$ and $g_J(-[C])=1$. Hence the argument for $a\ge 1$ in Proposition \ref{extreme} (where only adjunction formula is used) applies to $-[C]=aH-\sum b_iE_i$ to show that $g_J(-[C])=0$. This contradiction shows that we must have $a\le 0$, {\it i.e.}  $[C]\cdot H\ge 0$.
\end{proof}



\begin{prop}\label{+intk<9}
Suppose $M=\mathbb CP^2\#k\overline{\mathbb CP^2}$, $k< 9$. Then for any irreducible curve $C$ with $C^2\ge 0$, we have $K\cdot [C]< C^2$ and $SW([C])\ne 0$.
\end{prop}
\begin{proof}
Let $C=aH-\sum b_iE_i$. Applying Lemma \ref{int>0} to the curve $C$ and class $H$, we have $a\ge 0$. Furthermore, $C^2\ge 0$ implies $a>0$. Let $q:=C^2=a^2-\sum b_i^2\ge 0$. 

If $K\cdot [C]\ge C^2$, we have $3a+q\le \sum b_i$. We could assume $\sum b_i^2>0$, otherwise $C=aH$ would certainly have $K\cdot [C]<0< C^2$. By Cauchy-Schwartz inequality, we have
$$(q+3a)^2\le (\sum b_i)^2 \le k\sum b_i^2< 9(a^2-q).$$ This is equivalent to $$q^2+6aq+9q<0,$$which is impossible since we have $a>0$ and $q\ge 0$. This contradiction implies $K\cdot [C]< C^2$. 

We write it as $[C]\cdot (K-[C])<0$. If $SW(K-[C])\neq 0$, by SW=Gr we have a (possibly reducible) $J$-holomorphic curve in class $K-[C]$. Hence the irreducible curve $C$ must be a component of this curve and $C^2<0$. This contradicts to our assumption. Thus $SW(K-[C])=0$.

Since $\dim_{SW} [C]> 0$, we have wall-crossing formula
$|SW(K-[C])-SW([C])|=1$. This implies $SW([C])\neq 0$ because $SW(K-[C])=0$. 
\end{proof}

The next  result will be used in the proof of Theorem \ref{main1}.

\begin{prop}\label{k<8dec}
Suppose $M=\mathbb CP^2\#k\overline{\mathbb CP^2}$, $2\le k \le 7$. Then for any irreducible curve $C$ with $C^2\ge 0$, we have $[C]=[C_1]+[C_2]$ such that $SW([C_1])\cdot SW([C_2])\ne 0$.
\end{prop}
\begin{proof}
In fact, we will choose $[C_2]\in \mathcal E_K$. 

We first show that there exists $E\in \mathcal E_K$ such that $K\cdot ([C]-E)\le ([C]-E)^2$. Otherwise, for all $E\in \mathcal E_K$, we have $K\cdot ([C]-E)\ge ([C]-E)^2+2$. By Proposition \ref{+intk<9}, we have $K\cdot [C]< C^2$. Hence, we have 

\begin{equation}\label{.E>0}
2[C]\cdot E\ge C^2-K\cdot [C]>0, \, \, \forall E\in \mathcal E_K.
\end{equation}

If $[C]=aH-\sum b_iE_i$, we have $b_i>0$ by pairing it with $E_i$ using \eqref{.E>0}. Without loss, we assume $b_1\ge \cdots \ge b_k\ge 1$. Moreover, by pairing $[C]$ with $H-E_1-E_2$, we have \begin{equation}\label{basicE}
a>b_1+b_2>b_1 \ge \cdots \ge b_k\ge 1.
\end{equation}

When $2\le k\le 4$, we introduce $b_i=0$ for $k<i\le 4$. Applying \eqref{.E>0} to $E=H-E_1-E_2$, we have $2(a-b_1-b_2)\ge a^2+3a-\sum (b_i^2+b_i)$, which is $$a^2+a-b_1^2+b_1-b_2^2+b_2-b_3^2-b_3-b_4^2-b_4\le 0.$$On the other hand, the adjunction formula provides \begin{equation}\label{adj8}
a^2-3a-\sum (b_i^2-b_i)\ge -2.
\end{equation}
The above two inequalities imply $2a-b_3-b_4\le 1$. This contradicts to \eqref{basicE}.

When $5\le k\le 6$, we similarly introduce $b_i=0$ for $k<i\le 6$. Applying \eqref{.E>0} to $E=2H-E_1-E_2-E_3-E_4-E_5$, we have $$a^2-a-\sum_{i=1}^5(b_i^2-b_i)-b_6^2-b_6\le 0.$$ Along with adjunction formula \eqref{adj8}, we have $a-b_6\le 1$, which contradicts to \eqref{basicE}.

When $k=7$, we apply \eqref{.E>0} to $E=3H-2E_1-E_2-E_3-E_4-E_5-E_6-E_7$. This gives $$a^2-3a-b_1^2+3b_1-\sum_{i=2}^7(b_i^2-b_i)\le 0.$$ Along with \eqref{adj8}, we have $b_1=\cdots=b_7=1$. Since \eqref{basicE} implies $a\ge 3$, it contradicts to the last inequality. 

These contradictions guarantee there exists an $E\in \mathcal E_K$ such that $K\cdot ([C]-E)\le ([C]-E)^2$, which is equivalent to saying $\dim_{SW}([C]-E)\ge 0$. 

Since $(K-([C]-E))\cdot E=-2-[C]\cdot E<0$, we have $$(K-([C]-E))\cdot [C]=(K\cdot ([C]-E)-([C]-E)^2)+(K-([C]-E))\cdot E<0.$$
If $SW(K-([C]-E))\ne 0$, by SW=Gr we have a $J$-holomorphic subvariety in class $K-([C]-E)$. Hence the irreducible curve $C$ must be a component of this subvariety and $C^2<0$. This contradicts to our assumption $C^2\ge 0$. Thus $SW(K-([C]-E))=0$. By the wall-crossing formula $|SW(K-([C]-E))-SW([C]-E)|=1$, we have $SW(C-[E])\ne 0$. Hence the desired decomposition is $[C]=([C]-E)+E$.
\end{proof}

This conclusion of Proposition \ref{k<8dec} is not true for $k=8$ as we have seen in Example \ref{-K8blowup}. On the other hand, the above argument (for $k=7$) shows that this example is the only one where we do not have the desired decomposition when $k=8$.

Next, let us classify the negative irreducible curves with $a>0$ on $\mathbb CP^2\#k\overline{\mathbb CP^2}$ with $k<9$.

\begin{prop}\label{a>0}
Let $J$ be a tamed almost complex structure on $M=\mathbb CP^2\#k\overline{\mathbb CP^2}$, $k< 9$, and $C=aH-\sum b_iE_i$ be an irreducible curve with $C^2< 0$, $a>0$. Then $[C]$ is one of the following:
\begin{enumerate}
\item $H-\sum E_{k_j}$;
\item $2H-\sum E_{k_j}$;
\item $3H-2E_m-\sum_{k_j\neq m} E_{k_j}$;
\item $4H-2E_{m_1}-2E_{m_2}-2E_{m_3}-\sum_{k_j\neq m_i} E_{k_j}$;
\item $5H-E_{m_1}-E_{m_2}-\sum_{k_j\neq m_i} 2E_{k_j}$;
\item $6H-3E_{m_1}-\sum_{k_j\neq m_1} 2E_{k_j}$.
\end{enumerate} 
\end{prop}
\begin{proof}
Similar to  Proposition \ref{extreme}, we have
$$c^2+a^2=b_1^2+\cdots b_k^2,$$
$$-2+c^2+3a\le b_1+\cdots+b_k.$$

Now, $(c^2+3a-2)^2\le k(c^2+a^2)\le 8(c^2+a^2)$ holds by Cauchy-Schwartz inequality. This can be written as 
\begin{equation}\label{ac}
a^2-3a+(3c^2-9)a+(3a-12)c^2+c^4+4\le 0.
\end{equation}

The cases when $c^2<3$ ({\it i.e.} $-3<C^2<0$) actually follow from the classification of possible $-1$ and $-2$ sphere classes, see for example \cite{Man}.  More precisely, when $c^2=1$, the classification is obtained in Proposition 26.1, diagram (IV.8) of \cite{Man}. Especially, it contains our classes (3)-(6). When $c^2=2$, the classification is in diagram (IV.2) there. Let us reproduce the proof for readers'  convenience. When $c^2=1$, we have $$3a-\sum_{i=1}^r b_i=1, \quad a^2-\sum_{i=1}^r b_i^2=-1, \quad r<9.$$ By possibly adding some $b_i=0$ and $b_9=1$, we are reduced to solve $$3a-\sum_{i=1}^9 b_i=0, \quad a^2-\sum_{i=1}^9 b_i^2=-2, \quad b_9=1.$$ Rewriting the second equation, we have $$3a-\sum_{i=1}^9 b_i=0, \quad \sum_{i=1}^9 (a-3b_i)^2=18, \quad b_9=1.$$ In total there are three essentially different representations of $18$ as a sum of $9$ squares which are in the same residue class mod $3$: $$18=3^2+3^2+0^2+\cdots +0^2=(\pm 2)^2+(\pm 2)^2+(\pm 2)^2+(\mp 1)^2+\cdots (\mp 1)^2.$$ Up to the order of $b_i$, the solutions $(a;b_1, \cdots, b_9)$ are $$(3b; b+1, b-1, b, \cdots, b), \quad (3b\pm 2; b, b, b, b\pm 1, \cdots, b\pm 1).$$ Notice one of $b_i$ has to be $1$, this gives the list in our statement. 

Similarly, when $c^2=2$, our equations are reduced to $$3a-\sum_{i=1}^9 b_i=0, \quad \sum_{i=1}^9 (a-3b_i)^2=18,\quad b_9=0.$$ Up to the order of $b_i$, we have the same general solutions, but now with one of $b_i$ being $0$. These will also lie in the list.

Now let us assume $C^2\le -3$, {\it i.e.} $c^2\ge 3$. Then we have $a\le 3$ by \eqref{ac}.

When $a=1$ or $2$, then by adjunction, $b_i=1$ or $0$. This corresponds to our classes (1) and (2).

When $a=3$, we have $$\sum (b_i-1)b_i=2.$$ Hence only one $b_i$ could be $2$ or $-1$, others are $1$ or $0$. However, if $a=3$ and $b_i$ are $\pm1$ or $0$, and $k<9$, then $C^2>0$, a contradiction. Hence, it lies in class (3).
\end{proof}

Notice all $-1$ curve classes could be realized by complex structures. Hence Proposition \ref{a>0} says that all the negative curves with $a>0$ are obtained from the $-1$ curves by further blow ups and one can obtain other negative curves by blowing up up to $6$ more times. No other classes occur even for a tamed almost complex structure. 

The method used in the above proof could be extended to non-negative curves as well. We will prove the case for $C^2=0$ in the following. 

\begin{prop}\label{a>0'}
Let $J$ be a tamed almost complex structure on $M=\mathbb CP^2\#k\overline{\mathbb CP^2}$, $k< 9$, and $C=aH-\sum b_iE_i$ be an irreducible curve with $C^2=0$. Then $[C]$ is one of the following:
\begin{enumerate}
\item $H-E_{m_1}$;
\item $2H- E_{m_1}-\cdots -E_{m_4}$;
\item $3H-2E_{m_1}-E_{m_2}-\cdots -E_{m_6}$;
\item $4H-2E_{m_1}-2E_{m_2}-2E_{m_3}-E_{m_4}-\cdots E_{m_7}$;
\item $5H-E_{m_1}-2E_{m_2}-\cdots 2E_{m_7}$;
\item $6H-3E_{m_1}-3E_{m_2}-2E_{m_3}-\cdots -2E_{m_6}-E_{m_7}-E_{m_8}$;
\item $7H-3E_{m_1}-\cdots -3E_{m_4}-2E_{m_5}-\cdots- 2E_{m_7}-E_{m_8}$;
\item $5H-3E_{m_1}-2E_{m_2}-\cdots -2E_{m_4}-E_{m_5}-\cdots -E_{m_8}$;
\item $8H-3E_{m_1}-\cdots -3E_{m_7}-E_{m_8}$;
\item $4H-3E_{m_1}-E_{m_2}-\cdots -E_{m_8}$;
\item $8H-4E_{m_1}-3E_{m_2}-\cdots -3E_{m_5}-2E_{m_6}-\cdots-2E_{m_8}$;
\item $7H-4E_{m_1}-3E_{m_2}-2E_{m_3}-\cdots -2E_{m_8}$;
\item $9H-4E_{m_1}-4E_{m_2}-3E_{m_3}-\cdots -3E_{m_7}-2E_{m_8}$;
\item $11H-4E_{m_1}-\cdots-4E_{m_7}-3E_{m_8}$;
\item $10H-4E_{m_1}-\cdots-4E_{m_4}-3E_{m_5}-\cdots -3E_{m_8}$.
\end{enumerate} 
\end{prop}
\begin{proof}
First, by Proposition \ref{extreme}, all these $C$ are spheres. 

We can still use \eqref{ac}. If $C^2=0$, then $c=0$. Hence we have $0<a\le 11$.  
 By Lemma \ref{int>0}, we have $b_i\ge 0$.

Now we have $$3a-\sum_{i=1}^r b_i=2, \quad a^2-\sum_{i=1}^r b_i^2=0, \quad r<9.$$ If $r\le 7$, then by adding $b_8=b_9=1$, we have $$3a-\sum_{i=1}^9 b_i=0, \quad a^2-\sum_{i=1}^9 b_i^2=-2, \quad b_8=b_9=1.$$ This is exactly the same case as we have done for $C^2<0$. These are the cases (1)-(5) in our list.

Now we assume $r=8$ in the following. Without loss, we could assume $b_1\ge \cdots \ge b_8\ge 1$. We have a quick estimate $8b_8^2\le a^2\le 121$, thus $b_8\le 3$. If $b_8=1$, we change it to $b_8=2$ and add $b_9=1$. Hence we have 
$$3a-\sum_{i=1}^9 b_i=0, \quad a^2-\sum_{i=1}^9 b_i^2=-4, \quad b_8=2, \quad b_9=1.$$ Rewriting it, we have 
$$3a-\sum_{i=1}^9 b_i=0, \quad \sum_{i=1}^9 (a-3b_i)^2=36,\quad b_8=2, \quad b_9=1.$$

In total there are seven essentially different representations of $36$ as a sum of $9$ squares which are in the same residue class mod $3$ and sum to $0$: 

\begin{equation*} 
\begin{array}{lll}
 36&=&3^2+3^2+(-3)^2+(-3)^2+0^2+ \cdots +0^2\\
 &&\\
 &=&(\pm 4)^2+(\pm 1)^2+(\pm 1)^2+(\pm 1)^2+(\pm 1)^2+(\mp 2)^2+(\mp 2)^2+(\mp 2)^2+(\mp 2)^2\\
 &&\\
 &=&(\pm 5)^2+(\pm 2)^2+(\mp 1)^2+(\mp 1)^2+(\mp 1)^2+(\mp 1)^2+(\mp 1)^2+(\mp 1)^2+(\mp 1)^2\\
 \end{array}
 \end{equation*}

Up to the order of $b_i$, the solutions $(a;b_1, \cdots, b_9)$ are $$(3b; b+1, b+1, b-1, b-1, b, \cdots, b);$$  $$(3b\pm 1; b\mp 1, b, b, b, b, b\pm 1, b\pm 1, b\pm 1, b\pm 1);$$
$$(3b\pm 2; b\mp 1, b, b\pm 1, \cdots, b\pm 1).$$ 
 Note $b_8=2, b_9=1$, and $b_8$ has to change back to $1$, this gives the list (6)-(10) in our statement. 

If $b_8=2$, we change it to $b_8=3$ and add $b_9=1$. Hence we have 
$$3a-\sum_{i=1}^9 b_i=0, \quad a^2-\sum_{i=1}^9 b_i^2=-6, \quad b_8=3, \quad b_9=1.$$ Rewriting it, we have 
$$3a-\sum_{i=1}^9 b_i=0, \quad \sum_{i=1}^9 (a-3b_i)^2=54,\quad b_8=3, \quad b_9=1.$$

In total there are seven essentially different representations of $54$ as a sum of $9$ squares which are in the same residue class mod $3$ and sum to $0$: 

\begin{equation*} 
\begin{array}{lll}
 54&=&3^2+3^2+3^2+(-3)^2+(-3)^2+(-3)^2+ 0^2+ 0^2 +0^2\\
 &&\\
 &=&(\pm 5)^2+(\pm 2)^2+(\pm 2)^2+(\mp 1)^2+(\mp 1)^2+(\mp 1)^2+(\mp 1)^2+(\mp 1)^2+(\mp 4)^2\\
 &&\\
 &=&(\pm 4)^2+(\pm 4)^2+(\pm 1)^2+(\pm 1)^2+(\mp 2)^2+\cdots+(\mp 2)^2\\
 &&\\
 &=&(\pm 6)^2+(\mp 3)^2+(\mp 3)^2+0^2 \cdots +0^2.\\
 \end{array}
 \end{equation*}

Up to the order of $b_i$, the solutions $(a;b_1, \cdots, b_9)$ are $$(3b; b+1, b+1,b+1, b-1, b-1, b-1,b, b, b);$$ $$(3b\pm 2; b\mp 1, b, b, b\pm 1, \cdots , b\pm 1, b\pm 2);$$
$$(3b\pm 1; b\mp 1, b\mp 1, b, b, b\pm 1, \cdots, b\pm 1);$$ $$(3b; b\pm 2, b\mp 1, b\mp 1, b, \cdots, b).$$
 Note $b_8=3, b_9=1$, especially there are exactly one $1$ among all $b_i$, we know the first and the third in the above are impossible. After changing $b_8$ back to $2$, we have the list (11)-(13) in our statement. 

Finally, if $b_8=3$, we change it to $b_8=4$ and add $b_9=1$. Hence we have 
$$3a-\sum_{i=1}^9 b_i=0, \quad a^2-\sum_{i=1}^9 b_i^2=-8, \quad b_8=4, \quad b_9=1.$$ Rewriting it, we have 
$$3a-\sum_{i=1}^9 b_i=0, \quad \sum_{i=1}^9 (a-3b_i)^2=72,\quad b_8=4, \quad b_9=1.$$

In total there are twelve essentially different representations of $72$ as a sum of $9$ squares which are in the same residue class mod $3$ and sum to $0$: 

\begin{equation*} 
\begin{array}{lll}
 72&=&3^2+3^2+3^2+3^2+(-3)^2+(-3)^2+(-3)^2+ (-3)^2 +0^2\\
 &&\\
 &=&(\pm 8)^2+(\mp 1)^2+\cdots +(\mp 1)^2\\
 &&\\
 &=&(\pm 7)^2+(\pm 1)^2+(\pm 1)^2+(\pm 1)^2+(\mp 2)^2+\cdots+(\mp 2)^2\\
 &&\\
 &=&(\pm 6)^2+(\pm 3)^2+(\mp 3)^2+(\mp 3)^2+(\mp 3)^2+0^2 \cdots +0^2\\
 &&\\
 &=&6^2+(-6)^2+0^2 \cdots +0^2\\
 
  &&\\
 &=&(\pm 5)^2+(\pm 2)^2+(\pm 2)^2+(\pm 2)^2+(\mp 1)^2+(\mp 1)^2+(\mp 1)^2+(\mp 4)^2+(\mp 4)^2\\
  
  &&\\
 &=&(\pm 4)^2+(\pm 4)^2+(\pm 4)^2+(\mp 2)^2+\cdots+(\mp 2)^2.\\
 \end{array}
 \end{equation*}
 
 Up to the order of $b_i$, the solutions $(a;b_1, \cdots, b_9)$ are $$(3b; b+1, b+1,b+1, b+1, b-1, b-1,b-1, b-1, b), \quad (3b\pm 2; b\mp 2, b\pm 1, \cdots , b\pm 1);$$
$$(3b\pm 1; b\mp 2, b, b, b, b\pm 1, \cdots, b\pm 1), \quad (3b; b\pm 2, b\pm 1, b\mp 1, b\mp 1, b\mp 1, b, b, b, b);$$
$$
(3b; b+2, b-2, b, \cdots, b), \quad (3b\pm 2; b\mp 1, b, b, b, b\pm 1, b\pm 1, b\pm 1, b\pm 2, b\pm 2);$$
$$(3b\pm 1; b\mp 1, b\mp 1, b\mp 1, b\pm 1, \cdots, b\pm 1).$$

 Note $b_8=4, b_9=1$, especially there are exactly one $1$ and no $2$ among all $b_i$, we know that only the second and the third in the above are possible. After switching $b_8$ back to $3$, we have the list (14)-(15) in our statement. 
\end{proof}

After completing this paper, the author was kindly informed by Dusa McDuff and Felix Schlenk that they \cite{McSc} have also obtained precisely this list of classes in Proposition \ref{a>0'} as corresponding to all full symplectic packings of a $4$-ball by no more than $8$ balls. In fact, our Propositions \ref{a>0'} and \ref{corner01} altogether give an alternative argument of their full packing result. Precisely, the list of full packing corresponds to classes in $\partial\overline{\mathcal P}\cap \mathcal C_{M, K}$, which are the classes of $K$-symplectic spheres with self-intersection $0$ (the list of Proposition \ref{a>0'}) by Proposition \ref{corner01}. 

It is true that the negative curves in Proposition \ref{extreme} form the extremal rays of the curve cone.

\begin{prop}\label{dA}
The curve cone $A_J(M)$ of $(M=\mathbb CP^2\#k\overline{\mathbb CP^2}, J)$, $k< 9$, is a polytopic cone generated by the classes of spheres with non-positive self-intersections.
\end{prop}
\begin{proof}
By Proposition \ref{extreme}, we just need to prove that any irreducible curve $C$ with $C^2>0$ cannot span an extremal ray of the curve cone $A_J(M)$.

We look at the class $n[C]$. We know that when $k\le 8$, there are only finitely many elements in $\mathcal E_K$. This is because first there are finitely many classes in $\mathcal E_K$ such that its pairing with $H$ is non-negative ({\it e.g.} Proposition \ref{a>0}). Second, the class $H$ always pairs non-negatively with classes in $\mathcal E_K$ by Lemma \ref{-1a>0}.


Now we assume the maximal possible pairing of $C$ with elements in $\mathcal E_K$ is $l$. We choose $n$ large enough such that 
\begin{equation} \label{n2}
n^2C^2-n(K\cdot [C]+2l)-2>0, \quad n^2C^2-2n(K\cdot [C]+l)+6-k>0.
\end{equation} 
This is possible, since the coefficient of the quadratic term is positive.

Now we claim $[C]=\frac{1}{n}((n[C]-E)+E)$, $E\in \mathcal E_K$, gives a decomposition with both classes having nontrivial Seiberg-Witten. Hence $[C]$ is not extremal. 

First, we check the SW dimension
\begin{equation*} 
\begin{array}{lll}
 \dim_{SW}(n[C]-E)&=&(n[C]-E)^2-K\cdot (n[C]-E)\\
 &&\\
 &=&n^2C^2-nK\cdot [C] -2n [C]\cdot E-2\\
 &&\\
 &>&0\\
 \end{array}
 \end{equation*}
 
Take a symplectic form $\omega$ taming $J$. If $SW(K-(n[C]-E))\ne 0$, then $[\omega]\cdot (K-(n[C]-E))>0$. We also have $(K-(n[C]-E))^2>0$ by the second inequality of \eqref{n2}. On the other hand, by our assumptions, $[\omega]\cdot [C]>0$ and $C^2>0$. Hence by the light cone lemma, we have $(K-(n[C]-E))\cdot [C]\ge 0$, which contradicts to the first inequality of \eqref{n2}. Hence we have $SW(K-(n[C]-E))=0$. By wall crossing, $$|SW(n[C]-E)|=|SW(n[C]-E)-SW(K-(n[C]-E))|=1.$$ Apparently $SW(E)=1$. Hence $[C]=\frac{1}{n}((n[C]-E)+E)$ is not extremal.

Finally, we have classified all the classes of irreducible curves with non-positive self-intersections when $k<9$ in Proposition \ref{a<0} and \ref{a>0}. Especially, there are finitely many such classes. Hence our conclusion follows. In particular, our curve cone has no round boundary.
\end{proof}

This should be compared with Lemma \ref{-extremal}.

The next lemma basically shows that the phenomenon discovered in \cite{CP} for (elliptic) ruled surfaces cannot happen for small rational surfaces.

\begin{lemma}\label{dim0good}
Let $M=\mathbb CP^2\# k\overline{\mathbb CP^2}$. Then 
\begin{enumerate}
\item there is no irreducible curve $C$ such that $C^2, K\cdot [C]\ge 0$ when $k<9$.  
\item For $k=9$, the only such curve classes are $-mK=3mH-mE_1-\cdots-mE_9$, $m>0$.
\end{enumerate}
\end{lemma}

\begin{proof}
Let $[C]=aH-b_1E_1-\cdots-b_kE_k$ with $k\le 9$. The conditions then read as $$a^2\ge b_1^2+\cdots +b_k^2, \quad b_1+\cdots +b_k\ge 3a.$$
Since it is a $J$-holomorphic curve class for a tamed $J$, with $C^2\ge 0$, we have $a>0$ by the Lemma \ref{int>0}. Hence $$ka^2\ge k(b_1^2+\cdots +b_k^2)\ge (b_1+\cdots +b_k)^2\ge 9a^2.$$
It is a contradiction if $k<9$. When $k=9$ the equality holds if and only if $a=3m>0$ and $b_i=m$.
\end{proof}
Especially, it shows that there are no curve classes such that $\dim_{SW}([C])=0$ and $g_J([C])>0$ when $k<9$. There is a technical lemma for higher genus curve classes.

\begin{lemma}\label{small+}
Let $M=\mathbb CP^2\# k\overline{\mathbb CP^2}$, $k<9$. Let $C=aH-b_1E_1-\cdots-b_kE_k$ be a curve class. 
\begin{enumerate}
\item If  $C^2\ge 0$, then all classes with $a\le 2$ have $g_J(C)=0$.
\item If $g_J(C)=1$, then $C^2\ge 9-k$. The equality holds if and only if $C=3H-E_1-\cdots -E_k$.
\end{enumerate}
\end{lemma}
\begin{proof}
First $a>0$ as in last lemma. Let $C^2=c$, $g=g_J(C)$. Then by adjunction formula, the inequality $(b_1+\cdots +b_k)^2\le k(b_1^2+\cdots b_k^2)$ reads as $$(3a+2g-2-c)^2\le k(a^2-c).$$
When $a\le 2$, the inequality holds only if $g=0$ or $a=c=2, g=1, k=8$. But the latter case would imply $b_i=\frac{1}{2}$ which is a contradiction.

By Proposition \ref{extreme}, $C^2\ge 0$ if $g\ge 1$. Now we assume $g=1$ and $a \ge 3$. Then the inequality reads as $$(c-9)(c-(9-k))+(3-a)(6c-(3+a)(9-k))\le 0.$$ Hence if $c\le 9-k$, both terms are non-negative, with equality holding if and only if $a=3$ and $c=9-k$. 
This in turn implies $C^2=9-k$ if and only if $C=3H-E_1-\cdots -E_k$.
\end{proof}

 Let us define configuration of negative curves, as mentioned in the introduction, on a $4$-dimensional almost complex manifold $M$ with canonical class $K$. It includes the following data:

\begin{enumerate} [label=NC\theenumi]
\item A finite set $\mathcal{NC}=\{e_1, \cdots, e_n\}$ of distinct homology classes $e\in H_2(M, \mathbb Z)$ with negative self-intersection. Each of them individually could be realized by an irreducible $J$-holomorphic curve of genus $g_J(e)$ for some tamed $J$ with $K_J=K$. For $e\ne e'\in \mathcal{NC}$, we have $e\cdot e'\ge 0$.

\item A finite set of intersection points $P=\{p_1, \cdots, p_N\}$ labeled by integers. For each $e_i$, there is a collection of intersection points, called $X_{e_i}$ or $X_i$ for simplicity. The union $\cup X_i=P$.

\item For each $i$, and for each $j$ such that  $p_i\in X_j$, we associate a finite sequence of integers $(i, i_j^1, \cdots, i_j^{m_{ij}})$, satisfying the following conditions:  (i) $i_j^k>N$; (ii) $i_j^k<s_q^r$ if $r>k$; (iii) $i_j^k<s_q^k$ if $i<s$.  For any such sequence $S_{ij}$, there exists $j'$ such that the corresponding sequence $S_{ij'}\supset S_{ij}$.

\item For any intersection point $p_i\in X_j\cap X_{j'}$, the multiplicity  $M_{jj'}^i$ of $X_j$ and $X_{j'}$ at $p_i$ is defined as the length of the sequence  $S_{ij'}\cap S_{ij}$. Moreover, $\sum_i M_{jj'}^i=e_j\cdot e_{j'}$.


\item Two curve configurations are equivalent if they have the same set of $\mathcal{NC}$ and the same intersection data, probably after a relabelling (not necessarily keep the order). 
 

\end{enumerate}
 
 
In NC2, the intersection points $p_i$ should be viewed as abstract symbols, instead of actual points in $M$. The geometric intuition behind $S_{ij}$ in NC3 is the intersection data of two irreducible $J$-holomorphic curves. The first element records possible intersection points, the second element records possible tangent directions, etc..

By NC5, the order of the labels is not essential. The important information is whether the integers appeared in a given configuration are identical or not.

 In general, there might be infinitely many negative $J$-curves in a $4$-dimensional almost complex manifolds. However, we mainly focus on tamed almost complex structures on  $\mathbb CP^2\# k\overline{\mathbb CP^2}$ with $k<9$, where finiteness is guaranteed. Another point is the above definition only records the full information when the realizing configurations have smooth representatives for elements in $\mathcal{NC}$, which is also guaranteed in our situation. 
 
 One can check the negative curve configuration for  a tamed almost complex structure on $\mathbb{C}P^2\#k\overline{\mathbb{C}P^2}$ with $k<9$, {\it i.e.} the collection of all negative $J$-curves along with their intersection data, is a configuration of negative curves. In particular, the multiplicity $M_{jj'}^i$ is the degree of tangency of the curves $E_j$ and $E_{j'}$, realizing $e_j$ and $e_{j'}$, at a intersection point $p_i\in E_j\cap E_{j'}$.
 

Now we will focus on  rational $4$-manifolds $\mathbb CP^2\# k\overline{\mathbb CP^2}$ with fixed canonical class $K=-3H+E_1+\cdots +E_k$ (and $k<8$ in most situations in the following), we will look at the configurations of negative curves satisfying the following three properties: 

\begin{enumerate} [label=RS\theenumi]
\item all negative curves classes are chosen from the above classification and any two distinct ones intersect non-negatively; 
\item there is a rational cohomology class $a\in H^2(M, \mathbb Q)$ such that $a\cdot a> 0$ and it pairs positively with these negative curve classes and classes with non-trivial Seiberg-Witten;
\item  any $-1$ rational curve classes are integral positive linear combinations of these negative curve classes.
\end{enumerate} 

Recall the set of homology classes of negative curves in the configuration is denoted by $\mathcal{NC}$. We observe that there is at least one $-1$ rational curve classes in $\mathcal{NC}$ when the above three bullets are satisfied. For a $-1$ curve class $E$, we decompose it by negative curve classes $E=\sum a_iC_i$ as guaranteed by the third bullet. Pairing it with $K$, we know at least one $C_i$ has $K\cdot C_i<0$. Since $C_i^2<0$, it must be a $-1$ rational curve class. 

The negative curve configuration of a tamed almost complex structure is the collection of all irreducible $J$-holomorphic curves with negative self-intersection, with all the intersection information recorded. 

\begin{lemma}\label{tamedJ3bull}
The negative curve configuration for  a tamed almost complex structure on $\mathbb{C}P^2\#k\overline{\mathbb{C}P^2}$, $2\le k<8$, satisfies all the three properties RS1--RS3.
\end{lemma}
\begin{proof}
It is easy to see that the negative curve configuration for a tamed almost complex structure on  $\mathbb{C}P^2\#k\overline{\mathbb{C}P^2}$, $2\le k<8$, satisfies the first two  properties. The $a$ in the second bullet is chosen as the class of a taming symplectic form $\omega$ with rational cohomology class. We will show that the third bullet is also true. Since $SW(E)\ne 0$, we have $E=\sum a_iC_i$ where $C_i$ are the classes of irreducible curves and $a_i\in \mathbb Z^+$. If some $C_i^2\ge 0$, we can further get a nontrivial decomposition $C_i=\sum_j b_{ij}C_{ij}, b_{ij}\in \mathbb Z^+$, by Proposition \ref{k<8dec}. Here nontrivial means there are at least two $b_{ij}\ne 0$. If some $C_{ij}^2\ge 0$, we apply the same process again. This process will stop in finite steps since there is a positive lower bound of the pairing of $[\omega]$ with curve classes because $[\omega]$ is a rational cohomology class. When the process stops, we obtain the desired expression of $E$ as integral positive linear combination of classes in $\mathcal{NC}$.
\end{proof}

We summarize observations of the intersection number of negative curve classes.

\begin{lemma}\label{-1.-}
Let $M=\mathbb{C}P^2\#k\overline{\mathbb{C}P^2}$ and $\mathcal{NC}$ is the homology data in a configuration of negative curves satisfying RS1--RS3. 
\begin{enumerate}
\item When $k\le 6$, if the intersection number of a negative curves class with a $-1$ rational curve class is non-negative, it is either $0$ or $1$.  
\item When $k=7$, such intersection number could be $0$, $1$ or $2$. 
\item When $k=8$, such intersection number could be $0$, $1$, $2$ or $3$.
\end{enumerate}
\end{lemma}
\begin{proof}
We first prove for $k\le 6$. 
To see this, first,  we have $E\cdot H\ge 0$ by Lemma \ref{-1a>0}. Hence $-1$ curve classes are either $E_i$ or (1) and (2) in the list of Proposition \ref{a>0} when $k\le 6$. Since the Cremona transformation would move a $-1$ curve class to another $-1$ curve class, we could choose a Cremona transformation such that it moves a given negative curve class to the ones in Proposition \ref{a>0} or the first bullet of Lemma \ref{a<0}. Then Lemma \ref{a<0} and Proposition \ref{a>0}, as well as local positivity of intersections of $J$-holomorphic curves, gives the result by checking the mutual intersections of a negative curve with a different $-1$ rational curve in the list. 

Same argument applies to $k=7, 8$. Just observe when $k=7$, the negative curve class $3H-2E_1-E_2-\cdots-E_7$ starts to appear whose intersection with $E_1$ is $2$. When $k=8$, the class $6H-3E_1-2E_2-\cdots-2E_8$ starts to appear whose intersection with $E_1$ is $3$.
\end{proof}



The above lemma could be strengthened to the following.

\begin{lemma}\label{-.-}
Let $M=\mathbb{C}P^2\#k\overline{\mathbb{C}P^2}$ and $\mathcal{NC}$ is the homology data in a configuration of negative curves satisfying RS1--RS3. 
\begin{enumerate}
\item When $k\le 6$, the intersection number of any two elements in $\mathcal{NC}$ is no greater than $1$. 
\item When $k=7$, such intersection number is no greater than $2$. 
\end{enumerate}

\end{lemma}

\begin{proof}

We first prove for $k\le 6$. By Proposition \ref{extreme}, any element $e\in \mathcal{NC}$ has $g_{K}(e)=0$. Suppose there are two elements $e_1, e_2\in \mathcal{NC}$ with $e_1\cdot e_2\ge 2$. By Lemma \ref{a<0} and Proposition \ref{a>0}, we have $e_i\cdot H \le 2$. Moreover, any two negative curve classes with $e_i\cdot H\ge 1$ must have intersection number no greater than $1$ by checking the mutual intersections of cases (1) and (2) in Proposition \ref{a>0}. Hence, we have $(e_1+e_2)\cdot H\le 2$.



We calculate 
\begin{equation*} 
\begin{array}{lll}
 \dim_{SW}(-e_1-e_2)&=&(-e_1-e_2)^2-K\cdot (-e_1-e_2)\\
 &&\\
 &=&(e_1^2+K\cdot e_1)+(e_2^2+K\cdot e_2)+2e_1\cdot e_2\\
 &&\\
 &=&-4+2e_1\cdot e_2\\
 &&\\
 &\ge&0.\\
 \end{array}
 \end{equation*}
Hence, we can apply the wall crossing formula $|SW(-e_1-e_2)-SW(K+e_1+e_2)|=1.$ In the following, we will show $SW(K+e_1+e_2)=0$. 

If $SW(K+e_1+e_2)\ne 0$, then Lemma \ref{int>0}, applying to class $H$ and a K\"ahler structure $J_0$ such that $H$ is represented by a smooth $J_0$-holomorphic sphere,  implies $(K+e_1+e_2)\cdot H\ge 0$. However, since $(e_1+e_2)\cdot H\le 2$ in our situation, 
$$(K+e_1+e_2)\cdot H=-3+(e_1+e_2)\cdot H\le -3+2<0.$$
This contradiction implies that $SW(K+e_1+e_2)=0$ and thus $SW(-e_1-e_2)\ne 0$ by wall crossing. The equality  $0=e_1+e_2+(-e_1-e_2)$ contradicts to property RS2 since $a\cdot e_i>0$ and $a\cdot (-e_1-e_2)>0$. 
Hence, we must have $e_1\cdot e_2\le 1$ when $k\le 6$.

Same argument applies to $k=7$. Notice we could assume both curve classes $e_1, e_2$ have $e_i^2<-1$. Otherwise the statement follows from Lemma \ref{-1.-}. Then from the list of Proposition \ref{a>0}, $e_i\cdot H\le 2$ when $k=7$. Moreover, any two negative curve classes with $e_i\cdot H\ge 1$ must have intersection number no greater than $2$ by checking the mutual intersections of cases (1)-(3) in Proposition \ref{a>0}. 
Hence, $(e_1+e_2)\cdot H\le 2$ for $k=7$. Then the same argument for $k\le 6$ applies to show that $e_1\cdot e_2< 2$ when $k=7$. Along with Lemma \ref{-1.-}, we get the desired result.
\end{proof}

The intersection number of negative curve classes could be greater than $1$ when $k<6$. For instance, the intersection number of $H-E_1-E_2-E_3$ and $E_1+E_2+E_3-H$ is $2$. However, they cannot be in a configuration of negative curves simultaneously because of property RS2. 

The next lemma will be used in the proof of Theorem \ref{main1}.

\begin{lemma}\label{notriangle}
Let $M=\mathbb{C}P^2\#k\overline{\mathbb{C}P^2}$, $k<6$. 
If $e_1, e_2, e_3$ are different classes in $\mathcal{NC}$ of a configuration of negative curves satisfying RS properties, then at least one intersection $e_i\cdot e_j$, $i\ne j$, vanishes.
\end{lemma}

\begin{proof}
Suppose the conclusion is not true. By Lemma \ref{-.-}, we have $e_1\cdot e_2=e_2\cdot e_3=e_1\cdot e_3=1$. Hence, for the class $e=e_1+e_2+e_3$, we have $e^2=6+e_1^2+e_2^2+e_3^2\le 3$ and $g_K(e)=1$.

If $e\cdot H>0$, since the argument of Lemma \ref{small+} only uses adjunction formula, we have $e^2\ge 4$, which contradicts to our assumption $e^2\le 3$. 

Hence we can assume $e\cdot H\le 0$. We calculate the Seiberg-Witten dimension of $-e$:
 $$\dim_{SW}(-e)=(-e)^2-K\cdot (-e)=0.$$
Hence, we have the wall crossing formula $|SW(-e)-SW(K+e)|=1.$ In the following, we will show $SW(K+e)=0$. We have $$(K+e)\cdot H=-3+e\cdot H<0.$$ If $SW(K+e)\ne 0$, it would contradict to Lemma \ref{int>0}, when it is applied to class $H$ and a K\"ahler structure $J_0$ such that $H$ is represented by a smooth $J_0$-holomorphic sphere.

This contradiction implies that $SW(K+e)=0$ and thus $SW(-e)\ne 0$ by wall crossing. By RS2, the classes $e_i$ and $-e$ have positive intersection with a homology class $a$. This contradicts to $0=e_1+e_2+e_3+(-e)$.


Hence, at least one of the intersection $e_i\cdot e_j$, $i\ne j$, vanishes.
\end{proof}




To state the next result, we first define a {\it combinatorial blowdown} for a negative curve configuration.

\begin{definition}\label{defcombblowdown}
A (simple) combinatorial blowdown applied to a curve configuration is the following process of removing a homology class $E\in \mathcal{NC}$ with $E\cdot E=-1$ and $g_J(E)=0$. 
\begin{enumerate}[label=BD\theenumi]
\item Smoothly blow down the manifold $M$ to get a smooth manifold $M'$.
\item Remove $E$ and change any other classes $e\in \mathcal{NC}$ to $e'=e+(e\cdot E)E$ as a homology class in $H_2(M', \mathbb Z)$. Then remove all the new classes with $e'^2\ge 0$ from the modified $\mathcal{NC}$ and keep the rest if they still have negative square. Call the new set of homology classes obtained in this manner $\mathcal{NC}'$. Furthermore, we assume $e\cdot E\le 1$. 

\item Remove $p_i$ from $P$ if $p_i\notin X_{\alpha}\cap X_{\beta}$ for any $\alpha\ne \beta\in \mathcal{NC}'$.
\item The intersection data related to $E$ is changed as following. All points $p_i\in  X_E$ are removed from $P$, with a new point $x=p_0$ added if there exist $\alpha\ne\beta$ such that $X_{\alpha}\cap X_E, X_{\beta}\cap X_E\ne \emptyset$ and $e_{\alpha}', e_{\beta}'\in \mathcal{NC}'$. 

\item Suppose we have $p_i\in X_E$,  $p_i\in X_{\alpha}$ and the new relation $x\in X_{\alpha}'$. If there is a $\beta\ne \alpha$ such that  $e_{\beta}'\in \mathcal{NC}'$ and $S_{i\alpha}\subset S_{i\beta}$, then the sequence $S_{i\alpha}=(i, i_{\alpha}^1, \cdots, i_{\alpha}^{m_{i\alpha}})$ is changed to $S_{0\alpha}=(0, i, i_{\alpha}^1, \cdots, i_{\alpha}^{m_{i\alpha}})$. 

Otherwise, $S_{0\alpha}=(0, i, i_{\alpha}^1, \cdots, i_{\alpha}^{m'_{i\alpha}})$ where $(i, i_{\alpha}^1, \cdots, i_{\alpha}^{m'_{i\alpha}})$ is the maximal subsequence of $S_{i\alpha}$ contained in some $S_{i\beta}$ with $e_{\beta}'\in \mathcal{NC}'$. In particular, if there is no other $\beta$ such that $p_i\in X_{\beta}$ and $e_{\beta}'\in \mathcal{NC}'$, then  $S_{0\alpha}=(0)$.
\item All the other intersection data keep invariant.

\end{enumerate}
\end{definition}

When $k\le 6$, the last part of BD2 is guaranteed by Lemma \ref{-1.-}. 
It is straightforward to check that we have a configuration of negative curves after applying combinatorial blowdown, except an explanation for NC4. Assume $e_1', e_2'\in \mathcal{NC}'$.  Then $e_i\cdot E=0$ or $1$ for $i=1,2$. When $(e_1\cdot E)(e_2\cdot E)=0$, then $X_1\cap X_2=X_1'\cap X_2'$ and the intersection information is not changed for all these points. In particular, the multiplicities are kept invariant. When $e_i\cdot E=1$ for both $i=1,2$, we have two situations. In the first situation, we have $X_1'\cap X_2'=(X_1\cap X_2)\cup \{x\}$. Let $X_1\cap X_E=\{q_1\}$ and $X_2\cap X_E=\{q_2\}$. Moreover, the multiplicities $M^{q_1}_{1E}=M^{q_2}_{2E}=1$. 
Hence, we have the multiplicity of $X_1'$ and $X_2'$ at $x$ is also $1$. Since all the intersection data of $X_1\cap X_2$ is unchanged, we know the sum of the multiplicities is $e_1\cdot e_2 +1$ which is equal to  $e_1'\cdot e_2'$. In the second situation, we have $X_1\cap X_E=X_2\cap X_E=\{p_i\}$. If $S_{i1}\subset S_{i2}$, then $S_{01}\subset S_{02}=(0, S_{i1})$. Hence the multiplicity of $X_1'$ and $X_2'$ at $x$ will be the multiplicity of $X_1$ and $X_2$ at $p_i$ plus $1$. This again implies the sum of the multiplicities is $e_1\cdot e_2 +1$ which is equal to  $e_1'\cdot e_2'$.


\begin{prop}\label{combblowdown}
After combinatorially blowing down a $-1$ curve from a negative curve configuration with the three RS properties on $\mathbb{C}P^2\#k\overline{\mathbb{C}P^2}$, $3\le k<8$, we will have a configuration of negative curves with the three RS properties on $\mathbb{C}P^2\#(k-1)\overline{\mathbb{C}P^2}$.
\end{prop}
\begin{proof}
We fix a canonical class $K=-3H+E_1+\cdots+E_k$.  We have shown there is at least one $-1$ rational curve classes in $\mathcal{NC}$.  We combinatorially blow down this $-1$ curve. As we see as above, it is a configuration of negative curves.


Next we check properties RS1--RS3. It is direct to check that  RS1  holds for the new configuration. For RS2, we can choose $a'=a+(a\cdot E_k)E_k$. For RS3, we first write every class $E$ in $\mathcal E_{E_k^{\perp}}\subset \mathcal E_K$, {\it i.e.} those classes in $\mathcal E_K$ whose pairing with $E_k$ is zero, as $E=\sum a_iC_i$ with $C_i\in \mathcal{NC}$ and $a_i\in \mathbb Z^+$ by virtue of RS3 for the original configuration. Since $E\cdot E_k=0$, we have $\sum a_i(C_i\cdot E_k)=0$. Hence $$E=\sum a_iC_i=\sum a_iC_i'=E'.$$ If all $C_i'^2<0$, then we are done. If some $C_i'^2\ge 0$, since $C_i^2<0$, any $E\ne C_i\in \mathcal E_K$ has $E\cdot C_i$ is $0$ or 1 when $k\le 6$ and $0$, $1$ or $2$ when $k=7$. Hence we have $C_i'^2=0$ and $g_K(C_i')=0$ when $k\le 6$, or possibly $C_i'^2=3$ and $g_K(C_i')=1$ when $k=7$. The first type is Cremona equivalent to $H-E_1$, since we have a further decomposition which is Cremona equivalent to $H-E_1=(H-E_1-E_2)+E_2$.  That is, $C_i'$ is a sum of two $-1$ rational curve classes, which could be further written as a sum of classes in $\mathcal{NC}$ by RS3. The second type happens only when $C_i=3H-E_1-\cdots -E_6-2E_7$ and $C_i'=3H-E_1-\cdots -E_6$. We can write $C_i'=(2H-E_1-\cdots-E_5)+(H-E_5-E_6)+E_5$, a sum of classes in $\mathcal E_K$. Again these classes in $\mathcal E_K$ could be further written as a sum of classes in $\mathcal{NC}$ by RS3. Since some components $C_{ij}$ in the decomposition of $-1$ rational curve classes might have $C_{ij}'^2\ge 0$, we may take the process again. This process will stop in finite steps because of the rational class $a$ provided by RS2. When the process stops, we have the desired expression of $E$ as integral positive linear combination of classes in $\mathcal{NC}'$. 
\end{proof}

The statement of Theorem \ref{main1} for $\mathbb CP^2\# k\overline{\mathbb CP^2}$ with $k=1, 2$ or $S^2\times S^2$ is known. The claim for $k=1$ (and $S^2\times S^2$) is proved in \cite{LZ-generic} and $k=2$ is proved in Section 3. 

To prove the remaining cases, we use an induction argument, starting with $k=2$. Our induction hypothesis is that every negative curve configuration  on $\mathbb{C}P^2\#k\overline{\mathbb{C}P^2}$ satisfying RS1--RS3 is realized by a complex structure. For the base step $k=2$, it can either be shown by slightly modifying the argument in Section 3, or by a more direct argument as follows. By Lemma \ref{a<0} and Proposition \ref{a>0}, all the possible negative curve classses are $E_1, E_2, H-E_1-E_2, (1-s)H+sE_1, (1-s)H+sE_2,  (1-s)H+sE_1-E_2$ and $(1-s)H+sE_2-E_1$ with $s\ge 1$. By RS1 and RS3, $H-E_1-E_2$ and at least one of $E_i$, say $E_2$, is in $\mathcal{NC}$. If it is the case, other possible classes in $\mathcal{NC}$ are in $(1-s)H+sE_1$,  or $(1-s)H+sE_1-E_2$. There must be at least one such classes by applying RS3 to $E_1$. By RS1, it is direct to check that only one such class (and only one $s\ge 1$) would exist. These curve configurations are obtained by blowups of Hirzebruch surfaces.

The idea of our proof of Theorem \ref{main1} is to proceed by starting with a negative curve configuration on  $\mathbb{C}P^2\#k\overline{\mathbb{C}P^2}$ satisfying RS1--RS3, combinatorially blowing it down to get a negative curve configuration on $\mathbb{C}P^2\#(k-1)\overline{\mathbb{C}P^2}$, proving it also satisfies the three properties (which has been done in Proposition \ref{combblowdown}) and using the inductive hypothesis to realize this blown down configuration as a complex curve configuration, and then doing a complex blowup of this new configuration to realize the original curve configuration. Since Lemma \ref{tamedJ3bull} shows that the negative curve configuration for a tamed almost complex structure on $\mathbb{C}P^2\#k\overline{\mathbb{C}P^2}$,  $2\le k<8$, satisfies RS1--RS3, we will finish our proof of Theorem \ref{main1}.


Let us reformulate Theorem \ref{main1} in the following. 
\begin{theorem}\label{conf}
For rational $4$-manifolds $\mathbb CP^2\# k\overline{\mathbb CP^2}$ with $k<6$, the set of all the possible configurations of negative self-intersection curves for tamed almost complex structures is the same as the set for complex structures.
\end{theorem}
\begin{proof}



We have completed the first step of our proof in Proposition \ref{combblowdown}: to prove that combinatorially blowing down a $-1$ curve, say $E_k$, from a negative curve configuration with RS1--RS3 will give a configuration of negative curves with the same three properties. By our induction assumption, this is a complex curve configuration. 

Then we prove the second step: to show that we can reverse the process using complex blowups, {\it i.e.} we can apply a complex blowup of this new configuration to realize the original curve configuration, at least when $k<6$. Let us first review how the intersection data change after a combinatorial blowdown. First, since $e\cdot E_k\le 1$ by Lemma \ref{-1.-}, we know each class $e$ with $e'^2\ge 0$ is a $-1$ rational curve class and $e'$ is a square $0$ sphere class. Besides, all the points, along with their associated sequences to incidence relations, of $X_{e}$ with $e\in \mathcal{NC}\setminus \mathcal{NC}'$ are removed since there are no triple intersection points by Lemma \ref{notriangle}. Possibly a new intersection point $x\in P'$ might be added. Since the new configuration after blowdown is still a configuration of negative curves with RS1--RS3, we know $x$ is in exactly two sets $X_{\alpha}$, $X_{\beta}$, by Lemma \ref{notriangle}. Since $e\cdot E_k\le 1$ for all $e\in \mathcal{NC}$, in this case we must have $S_{0\alpha}=S_{0\beta}=(0)$, {\it i.e.} the multiplicity $M^0_{\alpha\beta}=1$.




Let us return to the induction process. By induction, the new configuration of negative curves is realized by a complex structure $J_0$ on $\mathbb CP^2\# (k-1)\overline{\mathbb CP^2}$. There are three cases that we do complex blow up differently. The first situation is for all $e'\in \mathcal{NC}'$, we have $e'^2=e^2$, {\it i.e.} $e\cdot E_k=0$. In this case, the elements in $\mathcal{NC}$ intersecting $E_k$ are all $-1$ classes, which would become square $0$ sphere classes after blowing down $E_k$. We check that these new square $0$ classes are nef, for which we only need to check its pairing with negative curves are all non-negative. These negative curves are exactly the ones in $\mathcal{NC}$ whose intersections with $E_k$ are $0$, {\it i.e.} the ones in $\mathcal{NC}'$. The nefness follows since different classes in $\mathcal{NC}$ intersect non-negatively and the new square $0$ class is $C'=C+E_k$ where $C$ is a $-1$ curve class in $\mathcal{NC}$ and $C\cdot E_k=1$.  Hence by Proposition 4.5 in \cite{LZ-generic}, for any given point of $\mathbb CP^2\# (k-1)\overline{\mathbb CP^2}$, we have a unique possibly reducible $J_0$-holomorphic rational curve in each of these new square $0$ nef classes passing through any given point. 
By Theorem 1.5 and Corollary 4.11 of \cite{LZrc}, reducible curves happen only when all components are in $\mathcal{NC}'$. Hence, for any square $0$ nef class $C'$ and for any point on $(\mathbb CP^2\# (k-1)\overline{\mathbb CP^2}, J_0)$ outside the negative curve locus, we have a smooth curve in class $C'$. Moreover, these square $0$ curves do not pass through any point in $P'$ and do not have mutual intersections with each other by Lemma \ref{notriangle}. 


Hence, we blow up at a point outside the negative curve locus of $(\mathbb CP^2\# (k-1)\overline{\mathbb CP^2}, J_0)$ would recover the original configuration of negative curves: all the classes in $\mathcal{NC}\setminus \mathcal{NC}'$ would become $-1$ curve classes, and no other new classes would be introduced since the intersection pairing of $E_k$ with any other irreducible negative curve is at most one by Lemma \ref{-1.-}; all the lost intersection data would be brought back since all the points in $P\setminus P'$ have multiplicity one and thus the associated data are $(p_i)$ at each of these points.

In the second situation, exactly one class $e'\in \mathcal{NC}'$ has $e'^2=e^2+1$. In this case, we just blow up $(\mathbb CP^2\# (k-1)\overline{\mathbb CP^2}, J_0)$ at a generic point of the negative curve in class $e'$. Then the argument is similar. Only square $0$ sphere classes would become new elements in $\mathcal{NC}$ after blowup. By the same argument in the first case, if a square zero class $C'$ is obtained from $C\in \mathcal{NC}$ by combinatorial blowdown, then $C'$ is nef. More precisely, we have $C'=C+E_k$ with square $-1$ classes $C, E_k\in \mathcal{NC}$. The conclusion follows from non-negativity of intersection of different classes in $\mathcal{NC}$ and $C\cdot E_k=1$ since classes in $\mathcal{NC}'$ are of type $A(+E_k)$ with $A\in \mathcal{NC}$. Moreover, any reducible representative of such a curve class $C'$ will have only negative square components by Theorem 1.5 and Corollary 4.11 of \cite{LZrc}. Such a representative cannot have $e'$ as a component since otherwise $(C'-e')\cdot C'+e'\cdot C'\ge 1$ by nefness of $C'$ and $e'\cdot C'=1$. It contradicts to $C'^2=0$. 


Hence, when we blow up  $(\mathbb CP^2\# (k-1)\overline{\mathbb CP^2}, J_0)$ at a generic point of the negative curve in class $e'$, the new negative curves are exactly those $-1$ curves in the original configuration which are obtained from square $0$ classes intersecting $e'$ (such an intersection number could only be $1$ by Lemma \ref{-.-}) as shown in last paragraph.  Thus we would recover the original configuration of negative curves. In particular, all the intersection points have multiplicity one and labeled differently. 

In the third situation, exactly two classes $e_1', e_2'\in \mathcal{NC}'$ have $e_i'^2=e_i^2+1$. In this case, the intersection point $x\in X_1'\cap X_2'$ is introduced when blowing down $E_k$. Both associated sequences to $x\in X_i'$ are $(0)$.  For convenience, we also call the corresponding intersection point in $(\mathbb CP^2\# (k-1)\overline{\mathbb CP^2}, J_0)$ by $x$. In this case, we blow up $(\mathbb CP^2\# (k-1)\overline{\mathbb CP^2}, J_0)$ at $x$.  All the square $0$ smooth curves cannot be tangent to each other or be tangent to negative curves, otherwise there will be triple intersections after blowup, contradicting Lemma \ref{notriangle}. Hence, point $x$ will become two points $p_1, p_2$ on $E_k$ along with other different points on $E_k$ which are intersection points of $E_k$ and other $-1$ classes. All these points have multiplicity one. Moreover, the new square zero classes in $\mathcal{NC}'$ obtained by $-1$ classes in $\mathcal{NC}$ have smooth representatives passing through $x$. First, by the argument in second situation, these square zero classes are nef and so any reducible representative has only negative-square components which cannot contain curves in class $e_1'$ or $e_2'$. And if such a representative passes through $x$, there would be a triple intersection of three negative curves that is forbidden by Lemma \ref{notriangle}. Hence,  we would recover the original configuration of negative curves by blowing up at $x$.

Finally, suppose there are three classes $e_1, e_2, e_3\in \mathcal{NC}$ with $e_i\cdot E_k=1$ such that $e_i'\in \mathcal{NC}'$. After combinatorially blowing down $E_k$, we will have a configuration with the three RS properties (by Proposition \ref{combblowdown}) such that  $e_i'\cdot e_j'=e_i\cdot e_j+1\ge 1$ for $i\ne j$. This contradicts to Lemma \ref{notriangle}. Hence the three cases above are exhaustive. 
\end{proof}

When $k=6$, there exists  $3$ mutually intersecting $-1$ curve classes. For example, $H-E_1-E_2, H-E_3-E_4, H-E_5-E_6$, or Cremona equivalently, $E_6, 2H-E_1-E_2-E_3-E_4-E_6, H-E_5-E_6$. There might be two possible configurations: all $3$ curves intersect at a single point, or the mutual intersections are different. When we blow down one of the $-1$ curves, we have two square $0$ curves, tangent to each other or intersecting at two different points respectively. To apply our induction argument in Theorem \ref{conf}, we need to show that the latter configuration is the generic one. 

When $k=7$, we have a new type of $-1$ curves, {\it i.e.} class (3) in Proposition \ref{a>0}. Thus there are possibly a negative curve and a $-1$ rational curve having intersection number $2$. After combinatorially blowdown, we will have a square $3$ class $C'=3H-E_1-\cdots -E_6$ (the corresponding $C$ is $3H-E_1-\cdots -E_6-2E_7$) to deal with. This is a class of $J$-genus $1$. Our induction argument works well to this single class, since one can show that a non-generic blow up has to be at a point on a negative curve, after a delicate analysis of the reducible curves in class $C$.

When $k=8$, we have  three more classes: (4)-(6) in Proposition \ref{a>0}. Hence there are negative curves with mutual intersection $3$. A similar argument should be enough to give a proof. The only difference is that a ``generic" blowup is no longer blowing up outside the negative locus. For example, $$6H-3E_1-2E_2-\cdots -2E_8=(3H-2E_1-E_2-\cdots -E_8)+(3H-E_1-\cdots-E_8).$$
Before blowing up $E_1$, both curves have positive squares (one nodal curve and one smooth curve intersect at the node). However, it is a non-generic phenomenon since it is a reducible curve. Another interesting new feature of this reducible curve is one of its component is of genus one although the original class is of genus $0$. Recall this cannot happen if the original class is $J$-nef by \cite{LZrc}. However, the following question still makes sense.
\begin{question}
Suppose $M$ is not diffeomorphic to $\mathbb CP^2\#k\overline{\mathbb CP^2}$. Let $E\in \mathcal E_{K_J}$. Is it true that for any subvarieties $\Theta=\{(C_i, m_i)\}$ in class $E$, {\it i.e.} $E=\sum m_i[C_i]$, we have $g_J([C_i])=0$?
\end{question}

For interested readers, the above decomposition of an exceptional class could also be seen from blowing down certain elliptic fibration of $E(1)\cong \mathbb CP^2\# 9\overline{\mathbb CP^2}$. We could make use of any elliptic fibration of $E(1)$ with an $I_2$ fiber and a $-1$ section $E$ (the existence of such fibration should be well known, see {\it e.g.} \cite{AZ}), but we will give an explicit construction in the following. 

We choose $p_0$ as a multiplication of a degree one polynomial and a degree two polynomial. Their zero set are two holomorphic spheres $C_1$ and $C_2$, in classes $H$ and $2H$, of general position in $\mathbb CP^2$. Let the two intersection points be $P$ and $Q$. Choose an irreducible degree $3$ polynomial $p_1$ such that the zero set $C_0$ is a smooth elliptic curve and its intersects $C_1$ and $C_2$ at $3$ and $6$ points respectively other than $P$ and $Q$. The family $t_0p_0+t_1p_1$ with $(t_0, t_1)\in \mathbb CP^1$ gives a pencil structure on $\mathbb CP^2$. Then we blow up the $3$ intersection points $C_0\cap C_1$ and $5$ out of the $6$ intersection points of $C_0\cap C_2$. After blowing ups, $C_0$ and $C_1$ becomes disjoint (for simplicity we still call them $C_0$ and $C_1$), in classes $3H-E_1-\cdots-E_8=-K$ and $H-E_1-E_2-E_3$ respectively. We notice $g_J([C_0])=1$ and $[C_0]+[C_1]=4H-2E_1-2E_2-2E_3-E_4-\cdots-E_8\in \mathcal E_K$ (it is class (4) in Proposition \ref{a>0}). It is easy to see that this decomposition is Cremona equivalent to the one right after Theorem \ref{conf}.

 
Theorem 1.1 in \cite{AGN} states that the inclusion of the space of compatible integrable complex structures into the space of all compatible almost complex structures is a weak homotopy equivalence for a rational ruled surface. Our Theorem \ref{conf} indicates that it may hold for $\mathbb CP^2\# k\overline{\mathbb CP^2}$ with $k<9$.

\subsection{Complex Configurations for small rational surfaces}\label{cp2+3}
From Theorem \ref{conf}, to know all the possible curve configurations for tamed almost complex structures, we only need to know that for complex structures when our underlying manifold is a small rational surface. This subsection summarizes all such possibilities for rational surfaces of Euler number no greater than $6$. Based on the all possible curve configurations for $\mathbb CP^2\#3\overline{\mathbb CP^2}$, we will discuss the limitation of the construction of almost K\"ahler forms using spherical classes as in \cite{LZ-generic}.


For $S^2\times S^2$, the possible types are Hirzebruch surfaces $\mathbb F_{2n}$. So the only irreducible negative curve is a $-2n$ curve which is in class $A-nB$ (or $B-nA$) where $A=[\{pt\}\times S^2], B=[S^2\times \{pt\}]$.

For $\mathbb CP^2\#\overline{\mathbb CP^2}$, the possible types are $\mathbb F_{2n+1}$. So the only negative curve is in class $(n+1)E-nH$.

For $\mathbb CP^2\#2\overline{\mathbb CP^2}$, we view this as blow up of $\mathbb F_{2n+1}$. We can either blow up at a point on the unique negative curve of $\mathbb F_{2n+1}$, or blow up at a point not on it. For the first case, our configuration is $E_2$, $H-E_1-E_2$ and $(n+1)E_1-nH-E_2$. For the latter case, our negative curves are $E_2$, $H-E_1-E_2$ and $(n+1)E_1-nH$. 

In all the above cases, the dual of the curve cone is the $J$-spherical cones $\mathcal S_J$ which is the K\"ahler cone. As we see in Theorem \ref{cp2+2cone} and in \cite{LZ-generic}, these configurations realize all the possible configurations of negative curves for any almost complex structures on $S^2\times S^2$, $\mathbb CP^2\#\overline{\mathbb CP^2}$ and $\mathbb CP^2\#2\overline{\mathbb CP^2}$.

For $\mathbb CP^2\#3\overline{\mathbb CP^2}$, it is a further blowup at certain points on some complex structure of $\mathbb CP^2\#2\overline{\mathbb CP^2}$. we can blow up
\begin{itemize}
\item at a point not on negative curves (a generic point), then the negative curves are $E_3$, $E_2$, $H-E_1-E_2$, $H-E_1-E_3$ and $(n+1)E_1-nH-E_2$, or $E_3$, $E_2$, $H-E_1-E_2$, $H-E_1-E_3$, $H-E_2-E_3$ (if $n=0$) and $(n+1)E_1-nH$;
\item at a generic point of $E_2$, then the curves are $E_3$, $E_2-E_3$, $H-E_1-E_2$ and $(n+1)E_1-nH-E_2$, or $E_3$, $E_2-E_3$, $H-E_1-E_2$ and $(n+1)E_1-nH$;
\item at a generic point of $H-E_1-E_2$, then the curves are $E_3$, $E_2$, $H-E_1-E_2-E_3$ and $(n+1)E_1-nH-E_2$, or $E_3$, $E_2$, $H-E_1-E_2-E_3$ and $(n+1)E_1-nH$;
\item at a generic point of $(n+1)E_1-nH-E_2$ or $(n+1)E_1-nH$, we get $E_3$, $E_2$, $H-E_1-E_2$, $H-E_1-E_3$ and $(n+1)E_1-nH-E_2-E_3$, or $E_3$, $E_2$, $H-E_1-E_2$, $H-E_1-E_3$ and $(n+1)E_1-nH-E_3$;
\item at the intersection point of $H-E_1-E_2$ and $(n+1)E_1-nH$, then the curves are $E_3$, $E_2$, $H-E_1-E_2-E_3$ and $(n+1)E_1-nH-E_3$;
\item at the intersection point of $(n+1)E_1-nH-E_2$ and $E_2$, then the curves are $E_3$, $E_2-E_3$, $H-E_1-E_2$ and $(n+1)E_1-nH-E_2-E_3$;
\item at the intersection point of $E_2$ and $H-E_1-E_2$, then the curves are only $E_3$, $E_2-E_3$, $H-E_1-E_2-E_3$ and $(n+1)E_1-nH-E_2$, or $E_3$, $E_2-E_3$, $H-E_1-E_2-E_3$ and $(n+1)E_1-nH$;
 
\end{itemize}

Notice in the last case, we only have one irreducible $-1$ rational curve, which is in class $E_3$. For all the others, we have at least two smooth $-1$ rational curves.
Then we can show that if for an almost K\"ahler structure the configuration of the negative curves is like the first six cases, the Nakai-Moishezon type theorem as Theorem \ref{NM2} holds since $J$-spherical cones are equal to the K\"ahler cones.

\begin{prop}
If the configuration of negative curves for an almost K\"ahler structure is one of the first six bullets listed above, we have 
$$\mathcal K_J^c=\mathcal S_J= \mathcal P_J=A_J^{\vee, >0}(M).$$
\end{prop}
\begin{proof}

First notice $ \mathcal P_J=A_J^{\vee, >0}(M)$. This is because, for any $J$, $A_J^{\vee, >0}(M)$ is contained in polytope with vertices $H$, $H-E_1$, $H-E_2$, $H-E_3$ and $2H-E_1-E_2-E_3$. 

Then notice $\mathcal S_J\subset \mathcal K^c_J$ and $\mathcal K^c_J\subset  \mathcal P_J$. Thus we could reduce the rest to show that  $\mathcal S_J= \mathcal P_J$. To prove $\mathcal S_J= \mathcal P_J=A_J^{\vee, >0}(M)$, we notice for the first case $A_J^{\vee, >0}(M)$ is another triangular bipyramid with all vertices are spherical classes. These vertices  could be represented or approximated by classes in $\mathcal S_J$. For the rest, they are all tetrahedra with at least two faces (thus span the tetrahedron) determined by $-1$ classes which can be generated by spherical classes.
\end{proof}

If the negative curves are $E_3$, $E_2-E_3$, $H-E_1-E_2-E_3$ and $(n+1)E_1-nH-E_2$, as in the last case, then the corners of $A_J^{\vee, >0}(M)$ are $(2n+3)H-(2n+1)E_1-E_2-E_3$, $(n+2)H-(n+1)E_1-E_2$, $(n+1)H-nE_1$ and $H-E_1$. The first one cannot be represented or approximated by a sphere, while the other three classes are all spheres. 
All the spherical classes are on the boundary (or more precisely, on the edges) of the triangular bipyramid with vertices $H$, $H-E_1$, $H-E_2$, $H-E_3$ and $2H-E_1-E_2-E_3$. While the class $(2n+3)H-(2n+1)E_1-E_2-E_3$ is in the interior of the hexahedron.

If the curves are $E_3$, $E_2-E_3$, $H-E_1-E_2-E_3$ and $(n+1)E_1-nH$, then the corners of $A_J^{\vee, >0}(M)$ are $(2n+2)H-2nE_1-E_2$ and other spherical classes. 

In other words, in both subcases of Case $7$, our spherical classes only span a face of the dual of curve cone. This is point the techniques in \cite{T2009, LZ-generic} does not work. However, this case will be covered by the inflation method in section \ref{infmethod}, see Theorem \ref{rr}.

\subsection{Configurations of smooth $-1$ rational curves}
In addition to study the possible configurations of all smooth negative curves, we could also look at the configurations of smooth $-1$ rational curves.

In this section, we assume $M_k=\mathbb CP^2\# k\overline{\mathbb CP^2}$ with $k\ge 1$. Dusa McDuff asks a couple of questions on possible numbers of $-1$ rational curves.

\begin{question}[McDuff]\label{mcduff}
\begin{enumerate}
\item What are the possible maximal numbers $l$ of disjoint embedded $-1$ rational curves for tamed almost complex structures on $M_k$? 
\item What are the possible numbers of embedded $-1$ rational curves for tamed almost complex structures on $M_k$, especially when $k\ge 9$? 
\end{enumerate}
\end{question}

We first give a few remarks on Question \ref{mcduff} (2). It is very direct to work out the possible numbers of embedded $-1$ rational curves for a complex structure on $M_k$ with $k<9$. For examples, from section \ref{cp2+3}, we know there could be $0$ or $1$ embedded $-1$ rational curves on $M_1$, $2$ or $3$ embedded $-1$ rational curves on $M_2$, $1$, $2$, $3$, $4$, or $6$ embedded $-1$ rational curves on $M_3$. Theorem \ref{main1} implies that at least when $k<8$, working with a tamed almost complex structure would not produce more possibilities. 

Now, let us work with the first part of Question \ref{mcduff} (1). Apparently, $l\le b^-(M_k)=k$. By Corollary \ref{-1}, we have $l\ge 1$ when $k\ge 2$. Our Theorem \ref{cp2+2} says that when $k=2$, there are at least two $-1$ rational curves. However, from the discussion in section  \ref{cp2+3}, the possible value of $l$ could be $1$ or $2$. The following result implies for a general $k\ge 1$, any integer $l$ with $1\le l\le k$ could be realized. This is based on an argument delivered to the author by McDuff.

\begin{prop}\label{rat-1}
Let $M_k=\mathbb CP^2\# k\overline{\mathbb CP^2}$ with $k\ge 3$ and $l$ an integer with $1\le l\le k$. There is an integrable (K\"ahler) $J$ on $M_k$ with exactly $l$ embedded $-1$ rational curves. Moreover, these rational curves are disjoint. 
\end{prop}
\begin{proof}
We start with a line on $\mathbb CP^2$. We blow up at $l$ distinct points on this line, call these exceptional curves $E_1, \cdots, E_l$. We then blow up at the intersection of $E_l$ with the line. Call the new exceptional curve $E_{l+1}$, and blow up its intersection with the line. Continue this process to do $k-l$ blow ups. Now we have a complex structure on $M_k$ which contains negative curves in classes $$H-E_1-\cdots-E_k, E_1, \cdots, E_{l-1}, E_l-E_{l+1}, \cdots, E_{k-1}-E_k, E_k.$$

If there is another embedded $-1$ rational curve, say $E$, in class $dH-\sum m_iE_i$, then we have $$d\ge m_1+\cdots +m_{k}, \hbox{ and } m_i\ge 0,\quad \forall i.$$ In particular, it implies $d\ge 0$. Moreover $d\ne 0$  otherwise all $m_i$ have to be $0$ and $E=0$.
 
Recall that a class  $x_0H-\sum x_iE_i$ is called {\it ordered} if $x_1\ge \cdots \ge x_k$. An ordered vector is {\it reduced} if $x_0\ge x_1+x_2+x_3$ and $x_i\ge 0$ for all $i$. Hence the class of $E$, when ordered, is a reduced class. However, by Lemma 3.4 of \cite{LBL}, there is no reduced class in $\mathcal E_K$. 
Hence $E_1, \cdots, E_{l-1}, E_k$ are the only exceptional curves, which are disjoint to each other. 
\end{proof}

We notice this result also gives a partial answer for Question \ref{mcduff} (2): when $k\ge 3$, there could be $1$ to $k$ embedded $-1$ rational curves.

There is a different construction for $l=1$ which is due to McDuff. 
By section 3.3, we have an integrable complex structure $J_2$ on $M_2$ with negative curves $$(n+1)E_1-nH-E_2, \quad H-E_1-E_2, \quad E_2.$$ Then blow up the intersection point of $H-E_1-E_2$ and $E_2$ to get $J_3$ on $M_3$. Then we blow up inductively the intersection point of $E_{i-1}-E_i$ and $E_i$ for $3\le i\le k$ to get $J_k$ on $X_k$ with negative curves $$(n+1)E_1-nH-E_2, \quad H-E_1-E_2-E_3, \quad E_i-E_{i+1}, i=2, \cdots, k-1, \quad E_k.$$
Use the similar argument as Proposition \ref{rat-1}, we can show that the only $-1$ rational curve lies in the class $E_k$. We leave the full detail to the interested readers.

Since we can choose any $n\ge 0$, there are infinitely many possible configurations for $l=1$ in Proposition \ref{rat-1}.

Proposition \ref{rat-1} could be extended to a general symplectic $4$-manifold. 

\begin{theorem}\label{disjoint-1}
Let $M$ be diffeomorphic to $N\#k\overline{\mathbb CP^2}$ with $k\ge 1$ and $N$ a minimal symplectic $4$-manifold. We assume $M$ is not diffeomorphic to one point blow up of an $S^2$ bundle over surface. Given $1\le l\le k$, there is a tamed $J$ on $M$ such that there are exactly $l$ embedded $-1$ rational curves. Moreover, these rational curves are disjoint.
\end{theorem}
\begin{proof}
When $M$ is a rational surface, it follows from Proposition \ref{rat-1}.

When $M$ is irrationally ruled, we could also choose such a K\"ahler structure. We start with $N=S^2\times \Sigma_h$ endowed with product complex structure. Denote the fiber class by $T$. Blow up at one point, we have two negative curves $E_1, T-E_1$. Further blow up $l-1$ times at distinct points on $T-E_1$ that are different from its intersection with $E_1$. Then we further blow up at the intersection of $E_l$ with $T-E_1-\cdots -E_l$. Call the new exceptional curve $E_{l+1}$, and blow up its intersection with $T-E_1-\cdots -E_{l+1}$. Continue this process to do $k-l$ blow ups. In total, we have done $k$ blow ups, and $M=N\#k\overline{\mathbb CP^2}$. Now we have negative curves in classes $$T-E_1-\cdots -E_k,  E_1, \cdots, E_{l-1}, E_l-E_{l+1}, \cdots, E_{k-1}-E_k, E_k.$$ For irrational ruled surfaces, $-1$ rational curves could only appear in classes $E_i$ and $T-E_i$ for $1\le i\le k$ (see {\it e.g.} Lemma 4.10 of \cite{LL}). However, $(T-E_i)\cdot (T-E_1-\cdots -E_k)=-1<0$ for any $i\le k$. Hence  $E_1, \cdots, E_{l-1}, E_k$ are the only exceptional curves, which are disjoint to each others. 

For a non-rational and non-ruled symplectic manifold, we could choose a tamed almost complex structure on $N$ such that in a small ball, it is integrable. Then we first blow up $l$ distinct points in this ball and then blow up consecutively on $E_l$. Again $E_1, \cdots, E_{l-1}, E_k$ are the only exceptional curves, which are disjoint to each others. 
\end{proof}




\subsection{Irrational ruled surfaces}
In this section, we discuss the cases of irrational ruled surfaces and prove Theorem \ref{ruled}. 

In general, the complex structures of non-rational ruled surfaces are much more complicated than that of rational ones. Any such minimal surface $M$ could be viewed as the projectivization $\mathbb P(E)$ of a vector bundle of dimension two over $\Sigma_g$. The curve cone behaves quite different when $E$ is unstable from it is semi-stable. When $E$ is unstable, {\it e.g.} $E=L\oplus \mathcal O$, the corresponding ruled surface $\mathbb P(E)$ has a negative curve. This is because by definition, we have a line bundle quotient $A$ of negative degree $a$. Then $C=\mathbb P(A)$ is an effective curve in the class $aT+U$ with $C^2=2a, 2a+1<0$. Recall that $T$ is the class of the fiber $S^2$ and $U$ is the class of a section with $U^2=0$ or $1$. In this case, the curve cone $A(M)$ is always closed. 

In contrast, when $E$ is semi-stable, the curve cone has different features. For convenience, we assume $E$ has even degree, and after twisting a line bundle we can then suppose $\deg E=0$. First it is always true that the nef cone is the same as the closure of the curve cone which is the first quadrant of the $U$-$T$ plane. This is because if there is an irreducible curve $C$ in the class $aT+bU$, then $C\in H^0(\mathbb P(E), \mathcal O_{\mathbb P(E)}(m)\otimes \pi^*A)=\Gamma(S^mE\otimes A)$ for some integer $m\ge 0$ and some line bundle $A$. It would imply $a\ge 0$ by semi-stability. On the other hand, $b\ge 0$ since there is always an irreducible curve in class $T$ and thus $[C]\cdot T\ge 0$. There is a famous example of Mumford showing that the curve cone might not be closed by the existence of the bundle $E$ over $\Sigma_g$ with $g>1$ such that $\Gamma(S^mE\otimes A)=0$ for all $m\ge 1$ whenever $\deg A\ge 0$. 

The above discussion suggests that bizarre things may happen for non-negative curves. See the discussion in the end of this section. However, the configuration of negative curves is always very simple.

\begin{theorem}\label{ruled}
For minimal irrational ruled surfaces, i.e. $S^2$ bundles over $\Sigma_{h\ge 1}$, the set of all the possible configurations of negative self-intersection curves for tamed almost complex structures are the same as the set for complex structures.
\end{theorem}
\begin{proof}
We divide our discussion in two cases.\\

$\bullet$ $\Sigma_h\times S^2$, $h\ge 1$ \\

In this case, let $U$ be the class of the base $\Sigma_h$ and $T$ be the class of the fiber $S^2$. Then the canonical class $K=-2U+(2h-2)T$. We suppose $F$ is an irreducible $J$-holomorphic curve with negative square, and $[F]=aU+bT$ for some integers $a$ and $b$. Then $a\cdot b<0$.

The adjunction formula tells us that
$$-2b+(2h-2)a+2ab=2g(F)-2.$$

If we project $F$ to the base $\Sigma_h$, the degree of the map is $a$. Since $\Sigma_h$ has genus at least one, we have 
$$2g(F)-2\ge a(2h-2).$$

Hence we have 
$$-2b+(2h-2)a+2ab\ge a(2h-2),$$
and in turn,
$$2b(a-1)\ge 0.$$

Since $a\cdot b<0$, it implies $a=1$ and $b<0$. For the configuration, we know that at most one class of the type $U-kT$ with $k\ge 0$ could appear because the negative intersection of each other.

On the other hand, we could also show that $U-kT$ is the class of some complex curve for a complex structure on $\Sigma_h\times S^2$. 
Suppose $L$ is a holomorphic line bundle with degree $2k\ge 0$. Then projectivization $\mathbb P(L\oplus \mathcal O)$ is topologically $\Sigma_h\times S^2$. Moreover, the section $S_{-k}=\mathbb P(L\oplus 0)$ of the $\mathbb P^1$ bundle has self-intersection $-2k$, which is in the class $U-kT$.\\

$\bullet$ Non-trivial $S^2$ bundles over $\Sigma_h$, $h\ge 1$ \\

Let $U$ be the class of a section with square $1$ and $T$ be the class of the fiber. Then the canonical class $K=-2U+(2h-1)T$. We suppose $F$ is an irreducible $J$-holomorphic curve with negative square, and $[F]=aU+bT$ for some integers $a$ and $b$. Then $a\cdot (a+2b)<0$.

The adjunction formula tells us that
$$-2b+(2h-1)a-2a+a^2+2ab=2g(F)-2 \ge a(2h-2),$$
which is equivalent to say that 
$$(a+2b)(a-1)\ge0.$$

This again implies $a=1$ and $b<0$, which shows the negative curves are in classes $U-kT$. 

The rest of the argument is exactly the same as the case of $\Sigma_h\times S^2$. Suppose $L$ is We only a holomorphic line bundle $L$ of degree $2k-1$. The section $S_{-k}=\mathbb P(L\oplus 0)$ is in the class $U-kT$.
\end{proof}

\begin{remark}
Notice that the Seiberg-Witten invariant calculation shows that there is a curve in class $aU+bT$, $a,b>0$ (let us focus on the trivial bundle case here, similar for the nontrivial bundle case), if and only if $ab+b+a-ah\ge 0$. This implies the closure of the curve cone always contains the first quadrant of the $U-T$ plane. However, it is intriguing to see whether there is a generic complex structure in the sense that only curve classes are the Seiberg-Witten non-trivial classes.
\end{remark}

We now give an interpretation of the example in \cite{CP}. Consider the nontrivial $S^2$ bundle over $T^2$. The classes $U$ and $T$ have the same meaning as above. Then the canonical class $K=-2U+T$. Consider the class $-2K$, it is the class of a square zero torus and its Seiberg-Witten dimension is $0$. Hence, generically we only have finitely many $J$-holomorphic curve in this class ($|SW(-2K)|=5$). The key observation of \cite{CP} is this is not true for complex structures: for any complex structures, there is always a $J$-holomorphic tori in class $-2K$ passing through any given point. Hence, after one blow-up at any point, we have a $-1$ $J$-holomorphic torus (possibly reducible) in class $4U-2T-E$. Notice its Seiberg-Witten dimension is negative, so generically there is no curve in this class. 

It is interesting to see that if we blow down along the other $-1$ curve, the one in class $T-E$, we will have $S^2\times T^2$.  The curve class is $4U+T$ in it (now our $U^2=0$), which is a genus $4$ class. Since the previous class in one point blow-up is represented as $4U+T-3E$ in our new basis, it implies every point is a triple point of a holomorphic curve in class $4U+T$ for any complex structures, which is of course not a generic phenomenon.

\section{Nakai-Moishezon duality for manifolds with abundant negative self-intersection curves}\label{infmethod}
In this section, we assume the tamed almost complex $4$-manifold $(M, J)$ has sufficiently many negative curves, such that $\mathcal P_J$ has no round boundary. We say there is no round boundary if the boundary is a cone over a polytope. Thus any class $e$ in the boundary of such $\mathcal P_J$  with $e^2=0$ should have $e\cdot C=0$ for some $C\in A_J(M)$.

As mentioned in the introduction, besides the subvarieties-current-form strategy, there is another way to attack Question \ref{dualC}. This is our main focus in this section. Alongside the main theorem in \cite{LZ},  we will need to construct $J$-tamed symplectic forms from an existing one. We use three operations in this section. The first one is the $J$-tamed inflation along curves with negative self-intersection (and sometimes along curves of square $0$), as described in Theorem \ref{buse} (Theorem \ref{mc} respectively). The second one is the summing of two $J$-tamed symplectic forms. The third one is rescaling, i.e. multiplying any $J$-tamed symplectic form with a positive number. The latter two make sense since $\mathcal K_J^t$ is a convex cone.

Let us begin with several lemmas. First, let us determine the polytopic boundary of $\mathcal P_J$.

\begin{lemma}\label{negbdy}
Assume that $h_J^+(M)=b^-(M)+1$ and $\mathcal P_J\ne \emptyset$.  
\begin{enumerate}
\item If $b^-(M)>1$, all the boundary hyperplanes are determined by negative curves.
\item If $b^-(M)=1$, all the boundary hyperplanes are determined by non-positive curves.
\item If $b^-(M)=0$, then $\mathcal P_J$ is a single ray.
\end{enumerate}
\end{lemma}
\begin{proof}
Recall $\mathcal P_J\subset  H_J^+(M)$, The third item is self-evident. Now we can assume $b^-(M)>0$.

By the light cone lemma, if $A$, $B$ are classes in $H_J^+$ with $A^2>0$, $B^2\ge 0$, then we have $A\cdot B>0$. This implies any positive curves cannot determine a polytopic boundary of $\mathcal P_J$.

If $A$, $B$ are classes in $H_J^+$ with $A^2, B^2\ge 0$, then we have $A\cdot B\ge 0$. And the equality holds if and only if $A$ is proportional to $B$. This implies any square $0$ curve classes will contribute a ray in the polygonal boundary of $\mathcal P_J$. It is a boundary hyperplane only when the cone $\mathcal P_J$ has dimension $2$, {\it i.e.} $b^-(M)=1$.

These give the proof of the first two facts.
\end{proof}

The next one is on the geometric property of a general $\mathcal P_J$. 

\begin{lemma}\label{closed}
Let $C_i$'s be the irreducible $J$-holomorphic curves in $A_J(M)$. If $C_i^2<0$, for any class $A\in  \mathcal P_J$, and any $0< \epsilon < \dfrac{A\cdot [C_i]}{-C_i^2}$, the class $(A+\epsilon [C_i])\in \mathcal P_J$. If $C_i^2\ge 0$, $A+\epsilon [C_i]\in  \mathcal P_J$ for any $\epsilon>0$.
\end{lemma}
\begin{proof}
First, $(A+\epsilon [C_i]) \cdot [C_i]= A \cdot [C_i](1+\epsilon \dfrac{C_i^2}{A\cdot [C_i]})>0$. 

When $i\neq j$, $(A+\epsilon [C_i]) \cdot [C_j]= A\cdot [C_j]+\epsilon [C_i]\cdot [C_j]>0$ because both terms are positive. 

Finally, $(A+\epsilon [C_i])^2= A^2+ \epsilon A\cdot [C_i] +(A+\epsilon [C_i]) \cdot [C_i]>0$.
\end{proof}

Among the $3$ operations mentioned above for constructing $J$-tamed symplectic form, the $J$-tamed inflation is the most important one. One of the most effective tools to determine the symplectic cone of a $4$-manifold is the (positive) symplectic inflation process introduced by Lalonde and McDuff in \cite{LaMc} along symplectic curves with non-negative self-intersection.  In \cite{LU}, this construction is extended to the case of negative self-intersection curves. 
 There is also a corresponding $J$-tamed version of it. McDuff, in \cite{Mc2}, proved the following result regarding the existence of (embedded) $J$-holomorphic curves with non-negative self-intersection.

\begin{theorem}[McDuff]\label{mc}
Let $J$ be a $\tau_0$-tame almost complex structure on a symplectic $4$-manifold $(M,\tau_0)$ that admits an embedded $J$-holomorphic curve $Z$ with $Z\cdot Z\ge 0$. Then there is a family $\tau_{\lambda}$, $\lambda\ge 0$, of symplectic forms that all tame $J$ and have cohomology class $[\tau_{\lambda}]=[\tau_0]+\lambda [Z]$. \end{theorem}

More recently, Buse (in \cite{Bus}) provided the corresponding version when $J$-holomorphic curves  with negative self-intersection are in presence. 

\begin{theorem}[Buse]\label{buse}
Fix a symplectic $4$-manifold $(M^4, J, \tau_0)$ such that $J$ is any $\tau_0$-tame almost complex structure. Assume that $M$ admits an embedded $J$-holomorphic curve $u:(\Sigma,j)\rightarrow (M^4,J)$ in a homology class $Z$ with $Z^2=-m$. For all $\epsilon>0$ there exist a family of symplectic forms $\tau_{\mu}$ all tame $J$ which satisfy
$$[\tau_{\mu}]=[\tau_0]+\mu Z$$
for all $0\le \mu \le \frac{\tau_0(Z)}{m}-\epsilon$. 
\end{theorem}

For the convenience of discussion, let us introduce the notion of the {\em formal $J$-inflation}.

\begin{definition}
An operation on a class $A$ is called a formal $J$-inflation along the cohomology class $C \in H_J^+(M)$ with $A\cdot C \ge 0$ and $C^2<0$, 
if $A$ is transformed to $A+\epsilon C$ with $0< \epsilon \le\dfrac{A\cdot C}{-C^2}$. When  $\epsilon=\dfrac{A\cdot C}{-C^2}$, we call it a maximal formal $J$-inflation.
\end{definition}

A self-evident fact for this definition is that a class obtained from formal $J$-inflation could be approximated by genuine $J$-tamed symplectic inflations if the class $A\in \mathcal K_J^t$ and $C$ is the class of an embedded $J$-holomorphic curve with $C^2<0$. This fact will be used frequently in this section along with the main result in \cite{LZ} that $\mathcal K_J^t\cap H_J^+(M)=\mathcal K_J^c$.

Lemma \ref{closed} demonstrates that the closure of the dual cone $ \mathcal P_J$ is closed under the operation of formal $J$-inflation. Because $\mathcal P_J$ is a convex cone, it is also closed under summing and rescaling. Thus $\overline{\mathcal P_J}$ is closed under all the three operations. Moreover, after the three operations, the class will still stay in the closure of the same connected component of $\mathcal K_J^t$ as beginning.

\begin{lemma}\label{lightcone} Suppose $h_J^+(M)=b^-(M)+1$. 
Let $C_1$ and $C_2$ be two smooth $J$-holomorphic curves with negative self-intersection, which provide two hyperplane pieces $\mathcal C_1$ and $\mathcal C_2$ of the boundary respectively. If the intersection $\mathcal C_1\cap \mathcal C_2\cap \overline{\mathcal P_J}\ne \emptyset$, then $([C_1]\cdot [C_2])^2\le C_1^2\cdot C_2^2$. Moreover, the equality holds if and only if there is a unique ray in the above intersection which is spanned by $[C_1]-\frac{[C_1]\cdot [C_2]}{C_2^2}[C_2]$.
\end{lemma}
\begin{proof}
Let us assume $[C_1]\cdot [C_2]\neq 0$ from now on. 
Suppose $([C_1]\cdot [C_2])^2> C_1^2\cdot C_2^2$ and there is a class $A\in \mathcal C_1\cap \mathcal C_2\cap \overline{\mathcal P_J}$. 

Hence we can construct a class $[C_2]-\frac{C_1\cdot C_2}{C_1^2}[C_1]$. Notice this class pairs non-negatively with all the curve classes. Moreover,
$$([C_2]-\frac{[C_1]\cdot [C_2]}{C_1^2}[C_1])^2=C_2^2-\frac{([C_1]\cdot [C_2])^2}{C_1^2}>0.$$

Since $A\cdot ([C_2]-\frac{[C_1]\cdot [C_2]}{C_1^2}[C_1])=0$ and $\mathcal P_J\subset H_J^+(M)$, we have $A^2<0$ by applying the light cone lemma to the $(1, b^-)$ space $H_J^+(M)$. This is a contradiction. 

The equality case goes similarly, except for the last step we have $A$ is proportional to $[C_2]-\frac{[C_1]\cdot [C_2]}{C_1^2}[C_1]$. Notice $[C_2]-\frac{[C_1]\cdot [C_2]}{C_1^2}[C_1]$ and $[C_1]-\frac{[C_1]\cdot [C_2]}{C_2^2}[C_2]$ span the same ray when the equality holds.
\end{proof}

\begin{example}
Take $[C_1]=E_1$ and $[C_2]=H-E_1-E_2$, then the intersection $\mathcal C_1 \cap \mathcal C_2$ is the ray of $H-E_2$.

On the other hand, when $[C_1]=E_8$ and $[C_2]=6H-3E_1-2E_2-\cdots -2E_8$, then there is no class $A$ in $\mathcal C_1 \cap \mathcal C_2\cap \overline{\mathcal P_J}$.
\end{example}

Next lemma describes what could be obtained if we only use the $J$-inflation along two negative curves alternatively. In the following, we say a class in $\overline{\mathcal P_J}$ is {\it achieved} by formal inflations, summing and rescaling if this class could be arbitrarily approximated by these three operations starting from a class in $\overline{\mathcal P_J}$.

\begin{lemma}\label{facetoedge}Suppose $h_J^+(M)=b^-(M)+1$. 
Let $C_1$ and $C_2$ be two smooth $J$-holomorphic curves with negative self-intersection, and denote the boundary of $\mathcal P_J$ determined by them as $\mathcal C_1$ and $\mathcal C_2$ respectively. If $\mathcal C_1\cap \mathcal C_2\cap \overline{\mathcal P_J}\ne \emptyset$, then starting from any class $A\in \overline{\mathcal P_J}$, we will achieve a class in $\mathcal C_1\cap \mathcal C_2\cap \overline{\mathcal P_J}$ by taking formal inflations along $C_1$ and $C_2$ alternatively. 
\end{lemma}

\begin{proof}
We may also assume the given class $A \in \mathcal C_1$, otherwise taking a maximal formal $J$-inflation along $C_1$.

We take maximal formal $J$-inflations along $C_1$ and $C_2$ alternatively.  Namely, we suppose $A_0=A$ and when $k\ge 0$, $$A_{2k+1}=A_{2k}+l_{2k+1}[C_2], \quad l_{2k+1}=\dfrac{A_{2k}\cdot [C_2]}{-C_2^2};$$
$$A_{2k+2}=A_{2k+1}+l_{2k+2}[C_1], \quad l_{2k+2}=\dfrac{A_{2k+1}\cdot [C_1]}{-C_1^2}.$$

By calculating the coefficients $l_k$ inductively, 
$$l_1=\dfrac{A\cdot [C_2]}{-C_2^2};$$
$$l_{2k+1}=l_1\cdot (\dfrac{([C_1]\cdot [C_2])^2}{C_1^2\cdot C_2^2})^k;$$
$$l_{2k}=l_1 \cdot \dfrac{[C_1]\cdot [C_2]}{-C_1^2} \cdot (\dfrac{([C_1]\cdot [C_2])^2}{C_1^2\cdot C_2^2})^{k-1}.$$

By Lemma \ref{lightcone}, we have $([C_1]\cdot [C_2])^2\le C_1^2\cdot C_2^2$. 

First let us assume $([C_1]\cdot [C_2])^2=C_1^2\cdot C_2^2$. To consider the convergence of the classes $A_k$ is indeed to consider the convergence of the corresponding rays of $A_k$. Simple calculation shows that $A_k$ approaches the ray of $[C_2]-\frac{[C_1]\cdot [C_2]}{C_1^2}[C_1]$, which is the (unique) intersection of $\mathcal C_1\cap \mathcal C_2$ in $\mathcal P$.

If we have $([C_1]\cdot [C_2])^2< C_1^2\cdot C_2^2$, we have the limit of $A_k$, whose value is $$\lim_{k\rightarrow \infty}A_k=A+\frac{l_1}{1-x}([C_2]-\frac{C_1\cdot C_2}{C_1^2}[C_1]),$$ where $x=\frac{(C_1\cdot C_2)^2}{C_1^2\cdot C_2^2}<1$. When we vary the class $A$, we get different limiting classes. It is straightforward to check that the pairing with $[C_2]$ is $$A\cdot [C_2]+\frac{l_1}{1-x}(C_2^2-\frac{([C_1]\cdot [C_2])^2}{C_1^2})=A\cdot [C_2]-\frac{A\cdot [C_2]}{1-x}(1-x)=0.$$ Similarly the paring with $C_1$ is zero as well. Since the formal inflation keeps our class in $\overline{\mathcal P_J}$, our conclusion follows.
\end{proof}

There is a better viewpoint to see the above calculations: we are actually doing formal inflation along the ray determined by $[C_2']=[C_2]-\frac{[C_1]\cdot [C_2]}{C_1^2}[C_1]$. Notice $[C_2']\cdot [C_1]=0$. Hence when we do inflation along the class $[C_2']$, the new class will keep orthogonality with $C_1$. When $C_2'^2=0$, we effectively inflate infinitely far along the ray through $C_2'$. When $C_2'^2<0$, the coefficient $\frac{l_1}{1-x}$ is nothing but the maximal inflation coefficient $\frac{A\cdot [C_2']}{-C_2'^2}$ by simple calculation.

\begin{lemma}\label{vertex}
Suppose $h_J^+(M)=b^-(M)+1$. Let $C_1$, $C_2$, $\cdots$, $C_n$ be smooth $J$-holomorphic curves with negative self-intersection, and denote the boundary of $\mathcal P_J$ determined by them as $\mathcal C_i$. Moreover, we assume  $\cap_i\mathcal C_i\cap \overline{\mathcal P_J}$ is a ray spanned by the class $B$. Then given any class $A\in  \overline{\mathcal P_J}$, one could achieve the class $B$ by taking formal $J$-inflations along $C_i$ (as well as summing and rescaling).
\end{lemma}
\begin{proof}
If there are two $C_i$'s, say $C_1$ and $C_2$, satisfy $([C_1]\cdot [C_2])^2=C_1^2\cdot C_2^2$, then by Lemma \ref{lightcone}, $\mathbb R^+B$ is the intersection $\mathcal C_1\cap \mathcal C_2\cap \overline{\mathcal P_J}$. Thus, maximal formal $J$-inflations along $C_1$ and $C_2$ would approach the ray $B$ as the limit as argued in Lemma \ref{facetoedge}. Hence we assume $([C_i]\cdot [C_j])^2<C_i^2\cdot C_j^2$ for $i\neq j$. We do induction for the $n$. When $n=2$, it is Lemma \ref{facetoedge}. 

Let us now show it for $n=3$, whose argument suggests the general induction step. We use the viewpoint after Lemma \ref{facetoedge}. We first find a class in $\mathcal C_1\cap \mathcal C_2$ by formally doing inflations for an arbitrary class $A\in \mathcal C_1\cap \mathcal P_J$ along  $[C_2']=[C_2]-\frac{C_1\cdot C_2}{C_1^2}[C_1]$ which is orthogonal to $[C_1]$. For the new class $A_1$, we apply formal inflations along $[C_3']=[C_3]-\frac{C_1\cdot C_3}{C_1^2}[C_1]$ which is also orthogonal to $[C_1]$. We thus obtain a class $A_2\in \mathcal C_1\cap \mathcal C_3$. Then we repeat this period-$2$ process. Notice $A_k\cdot C_1=0$ for all $k$. Hence, we are doing formal inflations along $[C_2']$ and $[C_3']$ alternatively. By the calculation as in Lemma \ref{facetoedge}, $A_k$ converges to $$A_1+\frac{l_1'}{1-x'}([C_3']-\frac{C_2'\cdot C_3'}{C_2'^2}[C_2'])$$ where $l'=\frac{A_1\cdot C_3'}{-C_3'^2}$ and $x'=\frac{(C_2'\cdot C_3')^2}{C_2'^2\cdot C_3'^2}$. By the viewpoint after Lemma \ref{facetoedge}, we are doing formal inflations along a single class $$[C_3'']=[C_3']-\frac{C_2'\cdot C_3'}{C_2'^2}[C_2'],$$ which is orthogonal to both $[C_1]$ and $[C_2']$, so that formal inflation along it preserves membership in $\mathcal C_1\cap \mathcal C_2$.

In general, given $C_1, \cdots, C_k$, we can find the classes $[C_2'], \cdots, [C_k']$ which are orthogonal to $C_1$ as above. Then we  start with $A\in \mathcal C_1\cap \mathcal P_J$, and use the inductive hypothesis to inflate along these classes $[C_i'], 2\le i\le k$ to get a class in $\cap_{i=2}^k\mathcal C_i'\cap \overline{\mathcal P_J}$. Since $[C_2'], \cdots, [C_k']$  are orthogonal to $C_1$ and $A\in \mathcal C_1\cap \mathcal P_J$, the limit class will remain in $\mathcal C_1$. Hence, the limit class is actually in $\cap_{i=1}^k\mathcal C_i\cap \overline{\mathcal P_J}$.
\end{proof}
%
%

\begin{example}
Suppose $M=\mathbb CP^2\#3\overline{\mathbb CP^2}$. Let the negative curves be $E_3$, $E_1-E_2$, $H-E_1-E_2-E_3$ and $E_2-E_3$. Then the first three classes will determine a intersection $2H-E_1-E_2$. Actually, we have $[C_1]=E_3$, $[C_2']=E_1-E_2$, $[C_3']=[C_3'']=H-E_1-E_2$. They are orthogonal to each other. Then our intersection ray is spanned by $$A+(A\cdot E_3) E_3+ \frac{A\cdot (E_1-E_2)}{2} (E_1-E_2)+A\cdot (H-E_1-E_2) (H-E_1-E_2).$$ 
We start with different classes in the closure of $\mathcal P_J$, we will get the same ray spanned by $2H-E_1-E_2$. 

For example, if we start with $H$, we get $2H-E_1-E_2$. If we start with $H-E_1$, we get $$H-E_1+\frac{1}{2}(E_1-E_2)=\frac{1}{2}(2H-E_1-E_2).$$
\end{example}

Now we are ready to prove Theorems \ref{noround} and \ref{rr}.
\begin{proof}(of Theorem \ref{noround})
As there is no round boundary, each connected component is a polytope (with perhaps infinitely many faces). We denote the known almost K\"ahler form by $\omega$. We want to prove that $\mathcal K_J^c$ is the connected component of $\mathcal P_J$ containing $[\omega]$, say $\mathcal P_J^{[\omega]}$. 


We first prove that  $\overline{\mathcal K_J^t}$ contains the closure of  $\mathcal P_J^{[\omega]}$. For each extremal ray, we could find boundary hyperplanes $\mathcal C_i$'s such that each $\mathcal C_i$ is determined by a smooth negative curve $C_i$. Thanks to Lemma \ref{vertex}, we could achieve this extremal ray by formal $J$-inflations along $C_i$'s, starting from the class $[\omega]$. After achieving each extremal ray, we could use summing and scaling to achieve  the closure of  $\mathcal P_J^{[\omega]}$.  Since any formal $J$-inflations could be infinitesimally approximated by $J$-tamed symplectic inflation processes, the cone generated by the classes of operating formal inflations is the closure of the the one such generated by $J$-tamed inflation. This implies $\overline{\mathcal K_J^t}$ contains the closure of  $\mathcal P_J^{[\omega]}$.

Since $\mathcal K_J^t\cap H_J^+(M)=\mathcal K_J^c$ and $\mathcal P_J\subset H_J^+(M)$, we have $\overline{\mathcal K_J^c}$ contains the closure of  $\mathcal P_J^{[\omega]}$. We notice both $\mathcal K_J^c$ and $\mathcal P_J^{[\omega]}$ are open convex cones in $H_J^+(M)$. For any element $x\in \mathcal P_J^{[\omega]}$, there is an open neighborhood of $x$ in it which is also contained in $\overline{\mathcal K_J^c}$. In particular, $x$ could be written as the sum of $\dim H_J^+$ linearly independent vectors in $\overline{\mathcal K_J^c}$. Since these $\dim H_J^+$ linearly independent vectors will span an open convex subcone $\mathcal A\subset \mathcal K_J^c$, we know  $x\in\mathcal A\subset \mathcal K_J^c$. This shows $\mathcal K_J^c$ contains $\mathcal P_J^{[\omega]}$.




Finally, it is easy to see $\mathcal K_J^c\subset \mathcal P_J$ and $\mathcal K_J^c$ is connected since it is a convex cone. Hence $\mathcal K_J^c$ is the connected component of $\mathcal P_J$ containing $[\omega]$.
\end{proof}

There are many tamed almost complex manifolds whose $\mathcal P_J$ has a round boundary.  For example, a generic almost complex structure on manifolds with $b^+>1$ will do since there are not enough curves. 
It is also true that a generic tamed almost complex structure on  $\mathbb CP^2\#k\overline{\mathbb CP^2}$ for $k>9$ also has a round boundary $\mathcal P_J$. First, by Lemma \ref{negbdy}, the boundary hyperplanes of $\mathcal P_J$ are determined by negative curves. On the other hand, only negative curves with non-trivial Seiberg-Witten are the classes of $-1$ rational curves since $2e^2=\dim_{SW}(e)+2g_J(e)-2\ge -2$. 
Hence, generically the only negative curves are $-1$ rational curves. However, the condition $e\cdot E>0$ for all $E\in \mathcal E_K$ does not guarantee that $e^2\ge 0$ for $k>9$, by taking $e=-K$ for instance. More precisely, there is an open subset of $\partial\mathcal P$ whose all elements pairing positively with $\mathcal E_K$  as in Remark \ref{k>9round}. See the new edition of \cite{McS} for more discussions.


On the other hand, Theorem \ref{rr} shows in some interesting cases, the assumptions of Theorem \ref{noround} are satisfied.
When $b^+=1$, there is yet another cone which is relevant to Question \ref{dualC}, the $K-$symplectic cone $\mathcal C_{M, K}$ introduced as \eqref{Kcone}.

\begin{proof}(of Theorem \ref{rr})
First a rational or ruled surface has $b^+=1$, hence $h_J^+=b^-+1$ holds. 
It is known that for rational surfaces $M=\mathbb CP^2\# k\overline{\mathbb CP^2}$ with $k<9$ or for minimal ruled surfaces, $\mathcal C_{M, K}$ has no round boundary. Actually, it is a polytope whose extremal rays are Cremona equivalent to $H$ or $H-E_1$ \cite{LBL}. 
Thus, $\mathcal P_J$ has no round boundary, since it is included in and cut along by more hyperplanes from the polytope $\mathcal C_{M, K}$. Moreover, it is connected. When $M=\mathbb CP^2\# k\overline{\mathbb CP^2}$ and $1<k<9$, the boundary of  $\mathcal P_J$ is constituted of hyperplanes determined by curves with negative self-intersection because of Lemma \ref{negbdy}. By Proposition \ref{extreme}, all these curves are smooth rational curves. By applying Theorem  \ref{noround}, we have $\mathcal K_J^c=\mathcal P_J$. Since $\{e|e\cdot E>0, \forall E\in \mathcal E_K\}\subset \mathcal P$ in this case (since $\mathcal C_{M, K}$ has no round boundary), we know $A_J(M)$ is a closed cone and $\mathcal P_J=A_J^{\vee, >0}(M)$.

For minimal ruled surfaces (and $\mathbb CP^2\# \overline{\mathbb CP^2}$), the boundary of (the closure of) the curve cone is constituted of two rays. One is the fiber class $T$. By a result of \cite{Mc2}, all the $J$-holomorphic curves in the fiber class $T$ are embedded. If there is a negative curve by assumption, then it is unique by Theorem \ref{ruled}. Hence the other ray is spanned by  an irreducible curve class $C$ with self-intersection $-n<0$. Moreover, by Theorem \ref{ruled}, the class $C\cdot T=1$. Hence for the boundary of $\mathcal P_J$, one of the ray is the fiber class $T$ and the other is $A$ with $A\cdot C=0$ and $A^2=n$.

As observed in \cite{Bus}, the class $C$ is also represented by an embedded curve. 
We do positive inflation along $T$ and negative inflation along $C$ (which determines the boundary A), $B+kT+lC$ with any $k>0$ and $0<l<\frac{B\cdot C}{n}$ spans $\mathcal P_J$.

Hence we have constructed $J$-tamed symplectic form in any any class of $\mathcal P_J$. Since $b^+=1$, we have $\mathcal K_J^c=\mathcal P_J$. By the above discussion, the boundary of $\overline{A}_J^{\vee, >0}$ contains two rays, one is the fiber class and the other is a class of non-negative square. Hence $\overline{A}_J^{\vee, >0}\subset \mathcal P$ and thus $\mathcal P_J=\overline{A}_J^{\vee, >0}$.

For the case of  $M=\mathbb CP^2\# 9\overline{\mathbb CP^2}$, we know $\mathcal P_J$ is almost polytopic in the sense that only the class $-K$ is possibly on the round boundary. In other words, any class in the interior of  $\mathcal P_J$ could be expressed as positive linear combinations of extremal rays. Moreover, $-K$ is the only possible curve class with non-positive self intersection which is not a rational curve. However, it does not contribute the vertices to  $\mathcal P_J$. Then by Lemma \ref{vertex}, we can achieve these extremal rays by formal $J$-inflations along smooth rational curves. Since formal $J$-inflations could be infinitesimally approximated by $J$-tamed inflations, we have shown $\overline{\mathcal K_J^c}$ contains the cone generated by non-negative linear combinations of extremal rays of $\mathcal P_J$. In particular $\overline{\mathcal K_J^c}\supset \mathcal P_J$. Since both $\mathcal K_J^c$ and $\mathcal P_J$ are open, we have $\mathcal K_J^c\supset\mathcal P_J$. Thus $\mathcal K_J^c=\mathcal P_J$, since the other direction of inclusion is trivial. Since $\{e|e\cdot E>0, \forall E\in \mathcal E_K\}\subset \overline{\mathcal P}$ in this case, we know $\mathcal P_J=\overline{A}_J^{\vee, >0}(M)$.

Finally, the statement $\mathcal K_J^c=\mathcal K_J^t$ follows from \cite{LZ}.
\end{proof}

In the case of minimal ruled surfaces, the curve cone does not necessarily be closed if the other extremal ray has square $0$. However, we have irreducible positive curve classes $A_n$ arbitrarily close to the boundary ray $\mathbb R^+ C$. The author does not know how to show there is always an embedded one.
If there is such one in each $A_n$,
by Theorem \ref{mc}, given any class $B$ in $\mathcal P_J$, we do (positive) inflations along $A_n$ and $F$ to obtain $B+k_1F+k_2A_n$ which spans all $\mathcal P_J$.

We remark that actually Theorem \ref{noround} has more applications. One important case is 
when we have a smooth representation of the anti-canonical class. In this case, all the negative curves are smooth rational curves with self-intersection $-1$ or $-2$. In fact, $\mathcal P_J$ is a polytope bounded by the hyperplanes determined by those rational curves and the curve in $-K$.

On a general four dimensional symplectic manifold, we do not usually have enough embedded $J$-holomorphic curves, although a generic almost complex structure on manifolds with $b^+=1$ does have so. Thus we have Theorem \ref{generic}.

\begin{proof}(of Theorem \ref{generic}) Now let us prove $\mathcal K_J^t=\mathcal C_{M,K}$  when $J$ is in a residual set of tamed almost complex structures. 

We first pick up any symplectic form $\omega$ with integral cohomology class. Let $e$ be an integral class in $\mathcal C_{M,K}$. There is an integer $L$ such that for all the integers $l>L$, $le-[\omega]$ and $le-[\omega]-K$ are both in $\mathcal C_{M,K}$. By Lemma 3.4 of \cite{LL}, $SW(le-[\omega])$ is nontrivial. Then $le-[\omega]$ could be represented by an embedded $J$-holomorphic curve for a generic $J$ tamed by $\omega$. We take the union of these residual subsets of $\omega$-tame almost complex structures and denote it by $\mathcal J_{e,\omega}$. By Theorem \ref{mc}, we know that $le=[\omega]+(le-[\omega])$ is represented by a $J$ tamed symplectic form for $J\in \mathcal J_{e,\omega}$. Take intersection of $\mathcal J_{e,\omega}$ for all integral class $e$, we get another generic subset $\mathcal J_{\omega}$ in all $\omega$-taming almost complex structures, since there are only countably many integral cohomology classes. Hence we have shown that for $J\in \mathcal J_{\omega}$, $\mathcal K_J^t=\mathcal C_{M,K}$. 

Because the set of symplectic forms that can be rescaled to be integral classes is dense in the space of symplectic forms, any tamed almost complex structure is tamed by a symplectic form with integral cohomology class. Taking union of $\mathcal J_{\omega}$ for all symplectic forms with integral cohomology class, we achieve our final residual subset $\mathcal J$ in all tamed almost complex structures. Hence $\mathcal K_J^t=\mathcal  C_{M, K}=\mathcal P_J$ for a generic tamed $J$.

Finally, by \cite{T2009} we have $\mathcal K_J^c\ne \emptyset$ for generic tamed $J$. Hence for all such $J$, $\mathcal K_J^t=\mathcal K_J^c$ by \cite{LZ}. Thus the proof of Theorem \ref{generic} completes.
\end{proof}

Notice the above proof is a renaissance of the argument in \cite{LL}. There is an alternative way to construct the residual set $\mathcal J$ using the strategy in \cite{T2009}. 

We endeavour to prove Question \ref{dualC} for all tamed $J$ rather than a residual subset. However, we may not have enough embedded $J$-holomorphic curves to apply the $J$-inflation even if we always have sufficient irreducible curves to play with the formal $J$-inflation.

\end{document}